\documentclass[10pt]{article} 
\usepackage{amsfonts,amsmath,latexsym}
\usepackage[T1]{fontenc}
\usepackage{hyperref}
\usepackage[dvips]{graphicx}
\usepackage{epsfig}
\usepackage{enumerate}
\usepackage{epsf}   
\usepackage{amsthm}
\usepackage{color}
\usepackage{amssymb}

\usepackage{caption}
\usepackage{subfigure}

\usepackage{fancyhdr} 

\usepackage{palatino}

\newtheorem{thm}{Theorem}[section]
\newtheorem{prop}[thm]{Proposition}
\newtheorem{lem}[thm]{Lemma}
\newtheorem{corro}[thm]{Corollary}
\newtheorem{defi}[thm]{Definition}
\newtheorem{rem}[thm]{Remark}

\newtheorem{ass}{Assumption}

\def\R{\mathbb R}
\def\N{\mathbb N}

\def\E{\mathbb E}
\def\P{\mathbb P}

\def\shb{{\cal B}}
\def\shc{{\cal C}}

\def\shf{{\cal F}}

\def\shp{{\cal P}}

\def\shu{{\cal U}}

\author{
{\sc Anthony LE CAVIL}
\thanks{ENSTA-ParisTech, Universit\'e Paris-Saclay.
 Unit\'e de Math\'ematiques Appliqu\'ees (UMA).
  E-mail:{ \tt anthony.lecavil@ensta-paristech.fr}} 
 {\sc,}\ {\sc Nadia OUDJANE}
\thanks{EDF R\&D,   and FiME (Laboratoire de Finance des March\'es de l'Energie
(Dauphine, CREST,  EDF R\&D) www.fime-lab.org).
E-mail:{\tt  
nadia.oudjane@edf.fr}}
\ {\sc and}\ {\sc Francesco RUSSO} 
\thanks{ENSTA-ParisTech, Universit\'e Paris-Saclay. Unit\'e de Math\'ematiques Appliqu\'ees (UMA). 
  E-mail:{\tt  francesco.russo@ensta-paristech.fr}}.
}

\date{August 2nd 2016}
\title{Particle system algorithm and chaos propagation related
 to  non-conservative McKean type stochastic differential equations.}

\newcommand{\MBFigure}[6]{
$\left. \right.$ \\
\refstepcounter{figure}
\addcontentsline{lof}{figure}{\numberline{\thefigure}{\ignorespaces #5}}
\begin{center}
\begin{minipage}{#1cm}
\centerline{\includegraphics[width=#2cm,angle=#3]{#4}}
\begin{center}
\upshape{F\textsc{ig} \normal
\end{center}
size{\thefigure}. $-$} #5
\end{center}
\label{#6}
\end{minipage}
\end{center}
$\left. \right.$ \\}

\begin{document}
\maketitle 
\begin{abstract} We discuss numerical aspects related to a new class of 
nonlinear Stochastic Differential Equations
in the sense of McKean, which are supposed to represent
 non conservative nonlinear Partial Differential equations (PDEs).
 We propose an original interacting particle system 
 for which we discuss the propagation of chaos. 
We consider a time-discretized approximation of this particle system to which
 we associate a random function which is proved to converge
to a solution of a regularized version of a nonlinear PDE.
\end{abstract}
\medskip\noindent {\bf Key words and phrases:}  
 Chaos propagation; Nonlinear Partial Differential Equations;
McKean type  Nonlinear Stochastic Differential Equations; Particle systems;
 Probabilistic representation of PDEs.

\medskip\noindent  {\bf 2010  AMS-classification}:
 65C05; 65C35; 68U20; 60H10; 60H30; 60J60; 58J35 
%%%%%%%%%%%%%%%%%%%%%%%%%%%%%%%%%%%%

\section{Introduction}
\label{SIntroduction}

%Framework and motivations
%%%%%%%%%%%%%%%%%%%%%%%%%%%%

Stochastic differential equations of various types are very useful to investigate nonlinear partial differential equations (PDEs)
at the theoretical and numerical level.
%Probabilistic representations of nonlinear Partial Differential Equations (PDEs) are interesting in several aspects.
 From a theoretical point of view,
% such representations allow for
%One one hand, 
they constitute probabilistic tools to study the analytic properties of the equation.
% (existence and/or uniqueness of a solution, regularity,\ldots).
Moreover they provide a microscopic interpretation of physical phenomena macroscopically drawn by a nonlinear PDE. 
%Similarly, stochastic control problems are a way of interpreting non-linear PDEs through Hamilton-Jacobi-Bellman equation that have their own theoretical and practical interests~(see~\cite{FlemingSoner}).
From a numerical point of view, such representations allow for extended Monte Carlo type methods, which are potentially less sensitive to the dimension of the state space. \\
Let us consider $d,p \in \N^{\star}$. Let 
%$\Phi, g, \Lambda:[0,T] \times \R^d \times \R \rightarrow \R$
$\Phi:[0,T] \times \R^d \times \R \rightarrow \R^{d \times p}$,
 $g:[0,T] \times \R^d \times \R \rightarrow \R^d$,
$\Lambda:[0,T] \times \R^d \times \R \rightarrow \R$,
 be Borel bounded  functions, $K: \R^d \rightarrow \R$ be a smooth mollifier in $\R^d$ and $\zeta_0$ be a probability  on $\R^d$. 
When it is absolutely continuous   
$v_0$ will denote its density so that $\zeta_0(dx) = v_0(x) dx$.  The main motivation of this work is the simulation of solutions to PDEs of the form 
\begin{equation}
\left \{ 
\begin{array}{l}
\partial_t v = \sum_{i,j = 1}^d \partial_{ij}^2 \big( (\Phi \Phi^t)_{i,j}(t,x,v) v \big) - div \big(g(t,x,v) v \big) + \Lambda(t,x,v)v \\
v(0,dx) = v_0(dx),
\end{array}
\right .
\end{equation}
through probabilistic numerical methods. Examples of nonlinear and 
nonconservative PDEs that are of that form arise  in 
hydrodynamics and biological
modeling. For instance one model related to  underground water flows 
 is known in the literature as the Richards equation
\begin{equation}
\left\{
\begin{array}{l}
\partial_t v = \Delta (\beta(v)) + div \big( \alpha(v) \big) + \phi(x) \\
v(0,\cdot) = v_0 \ ,
\end{array}
\right .
\end{equation}
where $\beta: \R \longrightarrow \R,  \
\alpha: \R \longrightarrow \R^d     $ and $\phi : \R^d \longrightarrow \R$.
Another example concerns biological mechanisms %which is present in biology  and describe phenomenons
as migration of biologial species or the evolution of a tumor growth. Such equations can be schematically written as 
%\begin{equation}
%\left\{
%\begin{array}{l}
%%\partial_t v = \Delta v + f(v,\nabla_x v) \\
%\partial_t v = \Delta v + \psi(v) \\
%v(0,\cdot) = v_0 \ .
%\end{array}
%\right .
%\end{equation}
%We can also consider
\begin{equation}
\left\{
\begin{array}{l}
\partial_t v = \Delta \beta(v) + f(v) \\
v(0,\cdot) = v_0 \ ,
\end{array}
\right .
\end{equation}
where $\beta : \R \longrightarrow \R$
%% OLD  $\beta : \R \longrightarrow \R^{d \times d}$
 is bounded, monotone and $f : \R \rightarrow \R$. This family of PDEs is sometimes called Porous Media type Equation with proliferation, due to the presence of the term $f$ that characterizes a proliferation phenomena and the term 
$\Delta \beta(v)$ delineates a porous media effect.
In particular, for $\beta(v) = v^2$ and $f(v) = v(1-v)$, this type of equation appears in the modeling of tumors. \\
%$ f(v) = v(1-v)$.
% measurable and for which several authors have studied existence/uniqueness of an associated probabilistic representation. } //
The present paper focuses on numerical aspects of a specific forward probabilistic representation initiated in \cite{LOR1},
relying on nonlinear 
SDEs in the sense of McKean~\cite{McKean}. In \cite{LOR1}, we have introduced and studied  a  generalized regularized McKean type nonlinear stochastic differential equation (NLSDE)
of the  form
%We are motivated in the simulation of a generalized regularized McKean type nonlinear stochastic differential equation (NLSDE)
\begin{equation}
\label{eq:NSDE}
\left\{
\begin{array}{l}
Y_t=Y_0+\int_0^t \Phi(s,Y_s,u(s,Y_s)) dW_s+\int_0^tg(s,Y_s, u(s,Y_s))ds\ ,\quad\textrm{with}\quad Y_0\sim  \zeta_0 \ ,\\
%u(t,\cdot) \quad  \textrm{is the density of} \ \nu_t \   \textrm{s.t. for any bounded continuous test function}\  \varphi \in \mathcal{C}_b(\R^d, \R)\\ 
u(t,y) = \E[K(y-Y_t) \, \exp \left \{\int_0^t\Lambda \big (s,Y_s,u(s,Y_s)\big )ds\right \}]\ ,\quad\textrm{for any}\ t\in [0,T]\ ,
\end{array}
\right .
\end{equation}
where the solution is the  couple process-function $(Y,u)$. The novelty with respect to classical McKean type equations
consists in the form of the second equation, where, for each $t>0$, in the classical case ($\Lambda = 0$) $u(t,\cdot)$ was explicitely given by the
 marginal law of $Y_t$.
 The present paper aims at proposing  and implementing a stochastic particle algorithm  to approximate \eqref{eq:NSDE} and 
investigating carefully its convergence properties. \\
\eqref{eq:NSDE} is the probabilistic representation of
the partial integro-differential equation  (PIDE) 
\begin{equation}
\label{epide}
\left \{
\begin{array}{l}
\partial_t \bar v = \frac{1}{2} \displaystyle{\sum_{i,j=1}^d} \partial_{ij}^2 \left( (\Phi \Phi^t)_{i,j}(t,x,K\ast \bar v) \bar v \right) - div \left( g(t,x,K\ast \bar v) \bar v \right) +\Lambda(t,x,K\ast \bar v) \bar v, \\
\bar v(0,x) = v_0 \ , %\mathcal{L}(Y_0) \ .
\end{array}
\right .
\end{equation}
in the sense that, given a solution 
$(Y,u)$ of~\eqref{eq:NSDE}, there is a solution $\bar v$
of \eqref{epide}
 in the
sense of distributions, such that 
 $u=K\ast \bar v:=\int_{\R^d}K(\cdot-y)\bar v(y)dy$. 
This follows, for instance, by a simple application of It\^o's formula,
as explained in Theorems 6.1 and 6.2, Section 6 in \cite{LOR1}.
%Besides the theoretical aspects related to the well-posedness of \eqref{eq:NSDE} considered in \cite{LOR1},
%our main motivation here is to simulate numerically efficiently their solutions.
%With this numerical objective,
% {\bf Ultimately, our motivation is to provide a probabilistic represention allowing to approximate numerically the limit PDE corresponding to the case where the smoothing kernel reduces to a Dirac. However, a theoretical analysis of this asymptotic limit is out of the scope of this paper, but will be investigated numerically via simulations reported at the end of this one.} 
Ideally our interest is devoted to \eqref{eq:NSDE} 
%and \eqref{epide} 
when the smoothing kernel $K$  reduces to a Dirac measure at zero.
To reach that scope, one would need to replace in previous equation  $K$ into $K_\varepsilon,$ where $K_\varepsilon$ converges to 
the Dirac measure and to analyze the convergence of the corresponding solutions. 
  However, such a theoretical analysis  is out of the scope of this paper, but it will be investigated numerically via simulations 
reported at the end.

In fact, in the literature appear several  probabilistic representations, with the objective
of  simulating numerically  the  corresponding PDE.
% each one having 
%specific features regarding the implied approximation schemes.\\
%
One method which has been largely investigated for
approximating solutions of time evolutionary PDEs is the method
of forward-backward SDEs (FBSDEs).  FBSDEs were initially developed
 in~\cite{pardoux}, see
also \cite{pardouxgeilo} for a survey and \cite{rascanu} for a recent monograph on the subject.  
The idea is to express the PDE solution $v(t,\cdot)$ at time $t$  as the expectation of
a functional of a so called forward diffusion process $X$, starting at time $t$.
%The
% second component of the process $\mathcalY$ is expressed as the expectation
%solution is a  functional of the whole path  of the forward process.
 Based on that idea, many judicious numerical schemes have been proposed 
%by~\cite{Zhang,BouchardTouzi,GobetWarin}. 
by~\cite{BouchardTouzi,GobetWarin}. 
However,  all those  rely on computing recursively conditional expectation functions which is known to be a difficult
 task in high  dimension. Besides, the FBSDE approach is \textit{blind} in the sense that the forward process $X$ is
 not ensured to explore the most relevant space regions to approximate efficiently the solution of the FBSDE of interest.
 On the theoretical side, the FBSDE representation of fully nonlinear PDEs still requires complex developments and is 
the subject of active research (see for instance~\cite{cheridito}). 
Branching diffusion processes 
%or in particular   superprocesses (see e.g. ~\cite{Dynkin})
 provide alternative  probabilistic representation of semi-linear PDEs,
 involving a specific form of non-linearity on the zero order term.
% We refer to~\cite{Dynkin} for the case of superprocesses. 
This type of approach has been recently extended in~\cite{labordere,LabordereTouziTan} to a
 more general class of non-linearities on the zero order term, with the so-called \textit{marked branching process}. 
One of the main advantages of this approach compared to FBSDEs is that it does not involve any regression computation
to calculate conditional expectations. 
A third class of numerical approximation schemes relies on McKean type representations. 
%The originality of the present paper is to combine together these two features 
%to relate the probabilistic system~\eqref{eq:NSDE} to the nonlinear 
%PIDE~\eqref{epide}.  
 In the time continuous framework, classical McKean representations are restricted to the conservative case ($\Lambda=0$).
Relevant contributions at the algorithmic level are \cite{BossyTalay95,
 BossyTalay97, bossyjourdain, mbajourdain}, 
and the survey paper \cite{Talay}.
% A third class of numerical approximation schemes relies on McKean type representations,
% but they were restricted to the conservative case ($\Lambda = 0$), see for example \cite{moral1}.
%  CITER DEL MORAL A CE NIVEAU.
In the case $\Lambda = 0$  with $g=0$, but with $\Phi$ possibly discontinuous,
some empirical implementations were conducted in \cite{BCR1, BCR3} in the one-dimensional and multi-dimensional case
respectively,  in order to predict the large time qualitative behavior 
of the solution of the corresponding PDE.

In the present paper we extend this type of McKean based numerical schemes
to the case of non-conservative PDEs ($\Lambda \neq 0$).
% One numerical intuition motivating our interest in (possibly non-conservative) 
%  PDEs representation of McKean type is the possibility to take 
% advantage of the  forward feature of this representation to bypass the dimension problem by localizing the particles
%  precisely in the \textit{regions of interest}, although this point will not be developed in the present paper.
 An interesting aspect  of this approach is that it is potentially able to represent fully nonlinear PDEs, by considering
a more general class of functions $\Lambda$ which may depend non-linearly not only on $u$ but
on its space derivatives up to the second order. This more general setting
% {\bf as well as the limiting case where the smoothing kernel converges to a Dirac},
 will be focused in a future work.
 % In this paper, for the considered class of generalized NLSDE, we establish the so called {\it propagation of chaos} of 
 % an associated interacting particle system and we develop a numerical scheme based on it. 
 % The convergence of this algorithm is proved by propagation of chaos  and through the
 %  control of the time discretization error.  Finally, some numerical simulations illustrate 
 % the practical interest of this new algorithm. 
% In the case $\Lambda = 0$  with $g=0$, but with $\Phi$ possibly discontinuous, 
% a first step in that direction in order to predict the qualitative behavior 
% of the solution for large time was numerically simulated
% in \cite{BCR1, BCR3} respectively for the one-dimensional and multi-dimensional case.
In the discrete-time framework, Feynman-Kac formula and various types of related particle approximation schemes were extensively 
analyzed in the reference books of Del Moral~\cite{DelMoral} and~\cite{moral1} but without considering the  specific case
of a  time continuous system~\eqref{eq:NSDE} coupled with  a weighting function $\Lambda$ 
which depends  nonlinearly on $u$.

By \eqref{eq:XIi} we introduce an interacting particle system associated to~\eqref{eq:NSDE}.
Indeed we replace one single McKean type stochastic differential equation 
with unknown process $Y$, with a system
of $N$ ordinary stochastic differential equations, whose solution consists in a system of
particles $\xi = (\xi^{j,N})$, 
 replacing the law of the process  $Y$ by the empirical mean law
 $\displaystyle{ S^N(\xi):=\sum_{j=1}^N\delta_{\xi^{j,N}} }$.
% is the  empirical measure of the particles $\xi = (\xi^{j,N})$.

%The main contribution of this paper is 
In Theorem \ref{CS8}  we  prove  the convergence  of the time-discretized particle system 
 under Lipschitz type assumptions on the coefficients $\Phi$, $g$ and $\Lambda$, obtaining an explicit rate.
The mentioned rate is based on the contribution of two effects. 
First,  the particle approximation error between the solution $u$ of \eqref{eq:NSDE}
and the approximation   $u^{S^N(\mathbf{\xi})}$, solution of
 \begin{equation} \label{NSDE3}
u^{S^N(\mathbf{\xi})}_t(y)= \displaystyle \frac{1}{N}\sum_{j=1}^N 
 K(y-\xi^{j,N}_t) \exp \left \{\int_0^t\Lambda \big (s,\xi^{j,N}_s,u^{S^N(\mathbf{\xi})} (s,\xi^{j,N}_s) \big ) ds\right \},
 %V_t\big (\xi^{j,N},u^{S^N(\mathbf{\xi})}(\xi^{j,N})\big ) }
 % u^m(t,y) = 
 % \int_{\mathcal C^d} 
 %  K(y- \omega_t)
 % \exp \left \{\int_0^t\Lambda \big (s,\omega_s,u^m(s,\omega_s) \big ) ds\right \} dm(\omega) \ 
 \end{equation}
%when $S^N(\xi)$ is the  empirical measure of the particles $\xi = (\xi^{j,N})$.
which is evaluated in Theorem \ref{TPC}.
The second effect
%to $u^{m_0}$,
 % the chaos propagation mentioned in Corollary \ref{CCP}
%of the continuous time interacting particle system 
is the  time discretization error,  established in Proposition \ref{prop:DiscretTime}. 
% that the propagation of chaos holds, under Lipschitz type assumptions on the coefficients $\Phi$, $g$ and $\Lambda$. 
%This is  the object of  Section \ref{SChaos}, 
% see  Theorem \ref{TPC} and subsequent remarks.
%Indeed the object of  Theorem \ref{TPC}  
% also states the convergence of the solution  $u^m$  of the 
% following linking equation
%  \begin{equation} \label{NSDE3}
%  u^m(t,y) = 
%  \int_{\mathcal C^d} 
%   K(y- \omega_t)
%  \exp \left \{\int_0^t\Lambda \big (s,\omega_s,u^m(s,\omega_s) \big ) ds\right \} dm(\omega) \ ,
%  \end{equation}
%  %where $m = m_Y$ is the law of $Y$ on the canonical space $\mathcal{C}^d$.
% when $m = S^N(\xi)$ is the  empirical measure of the particles to $u^{m_0}$, where $m_0$ is the law 
% of the solution of \eqref{eq:NSDE},
The errors are evaluated in the $L^p, p = 2, +\infty$ mean distance, in terms of the number $N$ of particles
and the time discretization step.
% In particular,  we provide estimates of the   convergence rates
%making use of  a refined analysis of the Lipschitz properties of $m \mapsto u^m$
% w.r.t.  various metrics on probability measures. 
One significant consequence of Theorem \ref{TPC} is Corollary \ref{CCP} which states the chaos propagation
of the interacting particle system. \\ 
 We emphasize that the proof of Theorem  \ref{TPC} relies on  Proposition \ref{PMoreGeneral}, whose formula
\eqref{eq:xiYuFinalGeneral} allows 
to control the particle approximation error without use of exchangeability assumptions on the particle system,
see Remark \ref{RMoreGeneral}.

%Paper Organization
%%%%%%%%%%%%%%%%%%%%%%%%%%%
The paper is organized as follows. After this introduction, 
we formulate the basic assumptions valid along the paper and recall important results proved in \cite{LOR1} and used in the sequel.
The evaluation of the particle approximation error is discussed in Section \ref{SChaos}.
 Section \ref{S8} focuses on the convergence of the time-discretized particle system.
Finally in Section \ref{SNum} we provide 
 numerical simulations  illustrating the performances of the
 interacting particle system in approximating the  limit PDE 
(i.e. when the smoothing kernel $K$ reduces to a Dirac measure at zero), %\eqref{epide}
in a specific case where the solution is explicitely known.
% The present paper is an improved version of the second part 
%of the unpublished paper \cite{LOR}.

\section{Notations and assumptions}

\label{S2}

\setcounter{equation}{0}

%%%%%%%%%%%%%%%%%%%%%%%%%%%%%%%%%%%
Let us consider $\shc^d:=\mathcal{C}([0,T],\R^d)$ metrized by the supremum norm $\Vert \cdot \Vert_{\infty}$, equipped with its Borel $\sigma-$ field $\mathcal{B}(\shc^d) = \sigma(X_t,t \geq 0)$ (and $\mathcal{B}_t(\shc^d) := \sigma(X_u,0 \leq u \leq t)$ the canonical filtration) and endowed with the topology of uniform convergence. $X$ will be the canonical process on $\shc^d$ and $\mathcal{P}_r(\shc^d)$ the set of Borel probability measures on $\shc^d$ admitting a moment of order $r \geq 0$. For $r=0$, $\mathcal{P}(\shc^d) := \mathcal{P}_0(\shc^d)$
 is naturally the Polish space (with respect to the weak convergence topology) of Borel probability measures on $\shc^d$ naturally equipped with its Borel $\sigma$-field $\mathcal{B}(\shp(\shc^d))$.
When $d=1$, we often omit it and we simply set 
 $\mathcal{C} :=  \mathcal{C}^1$.\\
We recall that the Wasserstein distance of order $r$ 
 and respectively the \textit{modified Wasserstein distance of order $r$}
for $r \ge 1$,
  between $m$ and  $m'$ in $\mathcal{P}_r(\mathcal{C}^d) $, denoted by $W^r_{T}(m,m')$ (and resp. $\widetilde{W}^r_{T}(m,m')$) 
are such that
\begin{eqnarray}
\label{eq:Wasserstein}
(W^r_{t}(m,m')) ^{r} & := & \inf_{\mu\in \Pi(m,m')} \left \{ \int_{\mathcal{C}^d \times \mathcal{C}^d} \sup_{0 \leq s\leq t} \vert X_{s}(\omega) - 
X_{s}(\omega')\vert ^{r}  d\mu(\omega,\omega')   \right \}, \ t \in [0,T] \ , \\
\label{eq:WassersteinTilde}
 ( \widetilde{W}^r_{t}(m,m')) ^{r} & := & \inf_{\mu\in \widetilde{\Pi}(m,m')} \left \{ \int_{\mathcal{C}^d \times \mathcal{C}^d} \sup_{0 \leq s\leq t} \vert X_{s}(\omega) - X_{s}(\omega')\vert ^{r} \wedge 1 \; d\mu(\omega,\omega')   \right \}, \ t \in [0,T] \ , 
\end{eqnarray}
where $\Pi(m,m')$ (resp. $\widetilde{\Pi}(m,m')$) denotes the set of Borel probability measures in $\shp(\shc^d \times \shc^d)$ 
with fixed marginals $m$ and $m'$ belonging to $\shp_r(\shc^d)$ (resp. 
%with fixed marginals $m$ and $m'$ belonging to 
$\shp(\shc^d)$ ).
In this paper we will use very frequently the Wasserstein distances
of order $2$. For that reason, we will simply set
$W_t: = W^2_t$ (resp. ${\tilde W}_t: = {\tilde W}^2_t$).
 \\
Given $N \in \N^{\star}$, $l \in \shc^d$, $l^1, \cdots, l^N \in \shc^d$, a significant role in this paper will be played by the Borel measures on $\shc ^d$ given by $\delta_l$ and $\displaystyle{ \frac{1}{N} \sum_{j=1}^N \delta_{l^j}}$. 
\begin{rem}
\label{RADelta}
Given $l^1,\cdots,l^N, \tilde{l}^1,\cdots, \tilde{l}^N \in \shc^{d}$, by definition of the Wasserstein distance we have, for all $t \in [0,T]$
\begin{eqnarray}
W_t \left(\frac{1}{N} \sum_{j=1}^N \delta_{l^j},\frac{1}{N} \sum_{j=1}^N \delta_{\tilde{l}^j} \right) \leq \frac{1}{N} \sum_{j=1}^N \sup_{0 \leq s \leq t} \vert l^j_s - \tilde{l}^j_s \vert^2 \nonumber \ .
\end{eqnarray}
\end{rem}
In this paper $\shc_b(\shc^d)$ denotes the space of bounded, continuous real-valued functions on $\shc^d$, for which the supremum norm is denoted by $\Vert \cdot \Vert_{\infty}$.
$\R^d$ is equipped with the scalar product $\cdot $ and $\vert x\vert $ stands for the induced Euclidean norm for $x \in \R^d$. Given two reals $a$, $b$ ($d=1$) we will denote in the sequel $a \wedge b := \min(a,b)$ and $a \vee b := \max(a,b)$. \\ 
$\mathcal{M}_f(\R^d)$ is the space of finite, Borel measures on $\R^d$. $\mathcal{S}(\R^d)$ is the space of Schwartz fast decreasing test functions and $\mathcal{S}'(\R^d)$ is its dual. $\mathcal{C}_b(\R^d)$ is the space of bounded, continuous functions on $\R^d$, $\mathcal{C}^{\infty}_0(\R^d)$ is the space of smooth functions with compact support. $\mathcal{C}^{\infty}_b(\R^d)$ is the space of bounded and smooth functions. $\mathcal{C}_0(\R^d)$ represents the space of continuous functions with compact support in $\R^d$. 
%$H^{-1}(\R)$ is the classical Sobolev space,
 $W^{r,p}(\R^d)$ is the Sobolev space of order $r \in \N$ in $(L^p(\R^d),||\cdot||_{p})$, with $1 \leq p \leq \infty$. \\
 $\mathcal{F}(\cdot): f \in \mathcal{S}(\R^d) \mapsto \mathcal{F}(f) \in \mathcal{S}(\R^d)$ will denote the Fourier transform on the classical Schwartz space $\mathcal{S}(\R^d)$ such that for all $\xi \in \R^d$,
 $$
 \mathcal{F}(f)(\xi) = \frac{1}{\sqrt{2 \pi}} \int_{\R^d} f(x) e^{-i \xi \cdot x} dx \ .
 $$ 
 We will designate in the same manner the corresponding Fourier transform on $\mathcal{S}'(\R^d)$.
 
\begin{def} \label{def_r.m}
For any Polish space $E$, we will designate by $\shb(E)$ its Borel $\sigma$-field. It is well-known that $\shp(E)$ is also a Polish space with respect to the weak convergence topology, whose Borel $\sigma$-field will be denoted by $\shb(\shp(E))$ (see Proposition 7.20 and Proposition 7.23, Section 7.4 Chapter 7 in \cite{BertShre}). \\
Let $(\Omega,\shf)$ be a measured space. A map 
$\eta : (\Omega,\shf) \longrightarrow (\shp(E),\shb(\shp(E)))$ will be
 called {\bf{random probability}}  (or {\bf{random probability kernel}}) if it is measurable. 
We will indicate by $\shp^{\Omega}(E)$ the space of  random probabilities.
%, i.e., the space of random probabilitis $\eta$ such that $\E[\eta(E)^2] < \infty$.
\end{def}
\begin{rem} \label{RMeasure}
Let $\eta : (\Omega,\shf) \longrightarrow (\shp(E),\shb(\shp(E)))$. $\eta$ is a random probability if and only if the two following conditions hold: 
\begin{itemize}
\item for each $\bar{\omega} \in \Omega$, $\eta_{\bar{\omega}} \in \shp(E)$,
\item for all Borel set $A \in \shb(\shp(E))$, $\bar{\omega} \mapsto \eta_{\bar{\omega}}(A)$ is $\shf$-measurable.
\end{itemize}
\end{rem}
This was highlighted in Remark 3.20 in \cite{crauel} (see also Proposition 7.25 in \cite{BertShre}).
\begin{rem} \label{RDelta}
Given $\R^d$-valued continuous processes $Y^1,\cdots, Y^n$, the application $\displaystyle{ \frac{1}{N} \sum_{j=1}^N \delta_{Y^j} }$ is a random probability on $\shp(\shc^d)$.
%\begin{proof}
In fact $\delta_{Y^j}, 1 \le j  \le N$ is a random probability by Remark \ref{RMeasure}.
%\end{proof}
\end{rem}
In  this article, the following assumptions will be used. 
\begin{ass}
\label{ass:main} 
\begin{enumerate}
	\item   $\Phi$ and  $g$ Borel functions defined on $[0,T] \times \R^d \times \R$
 taking values respectively in $\R^{d\times p}$ (space of $d \times p$ matrices) and $\R^d$ that are Lipschitz w.r.t. space variables: there exist finite positive reals $L_\Phi$ and $L_g$ such that for any $(t,y,y',z,z')\in [0,T] \times \R^d \times \R^d \times \R \times \R$, we have
$$
\vert \Phi(t,y',z')-\Phi(t,y,z)\vert \leq L_\Phi (\vert z'-z\vert + \vert y'-y\vert) \quad \textrm{and}\quad \vert g(t,y',z')-g(t,y,z)\vert \leq L_g (\vert z'-z\vert + \vert y'-y\vert)\ .
$$
%where for any $k\in\N^*$, $\vert \cdot\vert$ denotes the Euclidean norm on $\R^k$. 
	\item $\Lambda$ is a Borel real valued function defined on $[0,T]\times \R^d\times \R$  Lipschitz w.r.t. the
 space variables: %and the time variable : 
	there exists a finite positive real, $L_{\Lambda}$ such that  for any $(t,y,y',z,z')\in [0,T] \times \R^d\times \R^d\times \R \times \R$,
we have
 %$(t,t',y,y',z,z')\in [0,T]^2\times \R^d\times \R^d\times \R^2$:
$$
\vert \Lambda(t,y,z)-\Lambda(t,y',z')\vert \leq L_{\Lambda} (\vert y'-y\vert +\vert z'-z\vert )\ .
%\vert \Lambda(t,y,z)-\Lambda(t',y',z')\vert \leq L_{\Lambda} (\vert t-t'\vert+\vert y'-y\vert +\vert z'-z\vert )\ .
$$ 

%	\item $\Phi$, $g$ and $\Lambda$ are supposed to be uniformly bounded:  there exist finite positive reals $M_\Phi$, $M_g$ and $M_\Lambda$ such that, for any $(t,y,z)\in [0,T] \times \R^d \times \R$
%	
%	\begin{enumerate}
%	\item $$
%\vert \Phi(t,y,z)\vert \leq M_\Phi, \quad \vert g(t,y,z)\vert \leq M_g,$$
%\item $$ \vert \Lambda(t,y,z)\vert \leq M_\Lambda \ .
%$$
%\end{enumerate}
	\item $\Lambda$ is supposed to be uniformly bounded:  there exist a finite positive real $M_\Lambda$ such that, for any $(t,y,z)\in [0,T] \times \R^d \times \R$
	 $$ \vert \Lambda(t,y,z)\vert \leq M_\Lambda \ .
$$

	\item $K : \R^d \rightarrow \R_+$ is integrable, Lipschitz, bounded and whose integral is 1: there exist finite positive reals $M_K$ and $L_K$ such that for any $(y,y')\in\R^d\times \R^d$
$$
%0 \leq 
\vert K(y) \vert \leq M_K, \quad \vert K(y')-K(y)\vert \leq L_K \vert y'-y\vert\ \quad \textrm{and} \quad \int_{\R^d} K(x) dx = 1 \ .
$$
\item  $\zeta_0$ is a fixed Borel probability measure on $\R^d$ admitting a second order moment.
% We also add a sixth item.
% \begin{enumerate}
% \setcounter{enumi}{5}
\item The functions $s \in [0,T] \mapsto \Phi(s,0,0)$ and $s \in [0,T] \mapsto g(s,0,0)$ are bounded. $m_{\Phi}$ (resp. $m_g$) will denote the supremum $\sup_{s \in [0,T]} \vert \Phi(s,0,0) \vert$ (resp. $\sup_{s \in [0,T]} \vert g(s,0,0) \vert$). 
% \end{enumerate}
% \end{ass} 

\end{enumerate}
\end{ass}
Given a finite signed Borel measure $\gamma$ on $\R^d$,
$K * \gamma$ will denote the convolution function
$x \mapsto \gamma(K(x - \cdot))$. In particular if 
$\gamma$ is absolutely continuous with density 
${\dot \gamma}$, then
$(K * \gamma)(x) = \int_{\R^d} K(x-y) \dot \gamma(y)dy$. \\

To simplify  we introduce the following notations. 
 \begin{itemize}   
 	\item $V\,:[0,T] \times {\mathcal C}^d \times \mathcal{C}\rightarrow \R$ defined for any pair of functions $y\in \mathcal{C}^d$ and $z\in \mathcal{C}$, by 
\begin{equation}
\label{eq:V}
V_t(y,z):=\exp \left ( \int_0^t \Lambda(s,y_s,z_s) ds\right )\quad \textrm{for any} \ t\in [0,T]\ .
\end{equation}
	\item The real valued process $Z$ such that $Z_s=u(s,Y_s)$, for any $s\in [0,T]$, will often be denoted by $u(Y)$. 
\end{itemize}
With these new notations,  the second equation in~\eqref{eq:NSDE} can be rewritten as
\begin{equation}
\label{eq:mu}
 \nu_t(\varphi)=\E[(\check{K}\ast \varphi)(Y_t)V_t(Y,u(Y))]\ ,\quad \textrm{for any}\ \varphi\in \mathcal{C}_b(\R^d,\R)\ ,\\ 
\end{equation}
where  $u(t,\cdot)=\frac{d\nu_t}{dx}$ and $\check{K}(x) := K(-x)$.  
\begin{rem}
\label{rem:V}
Under Assumption \ref{ass:main}. 3.(b), $\Lambda$ is bounded.
% If  $\Lambda$ is supposed to be bounded, as it will always be the case, then $V$ is also bounded: 
Consequently
 \begin{equation}
 \label{eq:Vmajor1}
 0\leq V_t(y,z)\leq e^{tM_{\Lambda}}\ ,\quad\textrm{for any}\ (t,y,z)\in [0,T]\times \R^d\times \R\ . 
 \end{equation}
 Under Assumption \ref{ass:main}. 2.
  $\Lambda$ is Lipschitz. Then $V$ inherits in some sense this property. Indeed, observe that for any $(a,b)\in \R^2$, 
%\begin{eqnarray*}
 \begin{equation}        \label{EMajor1}
e^{b}-e^{a}=(b-a)\int_0^1 e^{\alpha b+(1-\alpha ) a}d\alpha 
\leq  e^{\sup (a, b)} \vert b-a \vert \ .
\end{equation}
Then for  any continuous functions $y,\,y'\in \mathcal{C}^d =\mathcal{C}([0,T],\R^d)$, and   $z,\,z'\in\mathcal{C}([0,T],\R)$, taking $a=\int_0^t\Lambda(s,y_s,z_s)ds$ and $b=\int_0^{t}\Lambda(s,y'_s,z'_s)ds$ in the above equality yields 
\begin{eqnarray}
\label{eq:Vmajor2}
\vert V_{t}(y',z')-V_t(y,z)\vert &\leq & e^{tM_{\Lambda}}\int_0^t\left \vert \Lambda(s,y'_s,z'_s)-\Lambda(s,y_s,z_s)\right \vert ds \nonumber \\
&\leq & e^{tM_{\Lambda}} L_{\Lambda} \, \int_0^t \left (\vert y'_s-y_s\vert +\vert z'_s-z_s\vert \right )ds \ .
\end{eqnarray}
\end{rem}
In Section \ref{S8}, Assumption \ref{ass:main}. will be replaced by what follows.
\begin{ass}
\label{ass:main2}
All items of Assumption \ref{ass:main}. are in force excepted 1. and 2. which are replaced by the following.
\begin{enumerate}
  \item There exist positive reals $L_{\Phi}$, $L_g$
%, $L_{\Lambda}$
 such that, for any $(t,t',y,y',z,z')\in [0,T]^2\times (\R^d)^2\times \R^2$,
    $$
    \vert \Phi (t,y,z)-\Phi (t',y',z')\vert \leq L_{\Phi}\,( \vert t-t'\vert^{\frac{1}{2}}+\vert y-y'\vert +\vert z-z'\vert ),
    $$
    $$
      \vert g (t,y,z)-g (t',y',z')\vert \leq L_{g}\,( \vert t-t'\vert^{\frac{1}{2}}+\vert y-y'\vert +\vert z-z'\vert) .
      $$
  \item There exists a positive real $L_{\Lambda}$ such that, for any $(t,t',y,y',z,z')\in [0,T]^2\times (\R^d)^2\times  \R^2$,
  $$
  \vert \Lambda (t,y,z)-\Lambda (t',y',z')\vert \leq L_{\Lambda}\,( \vert t-t'\vert^{\frac{1}{2}}+\vert y-y'\vert +\vert z-z'\vert ).
  $$
\end{enumerate}
% We also add a sixth item.
% \begin{enumerate}
% \setcounter{enumi}{5}
% \item The functions $s \in [0,T] \mapsto \Phi(s,0,0)$ and $s \in [0,T] \mapsto g(s,0,0)$ are bounded. $m_{\Phi}$ (resp. $m_g$) will denote the supremum $\sup_{s \in [0,T]} \vert \Phi(s,0,0) \vert$ (resp. $\sup_{s \in [0,T]} \vert g(s,0,0) \vert$). 
% \end{enumerate}
 \end{ass} 

We end this section by recalling important results established in 
 our companion paper \cite{LOR1}, for which Assumption \ref{ass:main}.
 is supposed to be satisfied.
 Let us first remark that the second equation of \eqref{eq:NSDE} can be 
rewritten as
 \begin{eqnarray} \label{Eum}
 u^m(t,y) = \int_{\shc^d} K(y-\omega_t) \exp \Big \{ \int_0^t 
\Lambda(s,\omega_s, u^m(s,\omega_s)) \Big \} dm(\omega), \; (t,x) \in [0,T] \times \R^d,
 \end{eqnarray}
 with $m = m_Y$ being the law of the process $Y$ on the canonical space $\shc^d$. \\
Indeed for every $m \in \shp(\shc^d)$, Theorem 3.1 of \cite{LOR1}
shows that equation \eqref{Eum} is well-posed and so it properly defines
a function $u^m$. 
The lemma below, established in Proposition 3.3 of \cite{LOR1}, states stability results on the function $(m,t,y) \mapsto u^m(t,y)$.
\begin{prop}
\label{lem:uu'}
We assume the validity of items $2.$, $3$ and $4.$ of Assumption \ref{ass:main}. \\%Let $u$ be a solution of~\eqref{eq:NSDE}.
The following assertions hold. 
\begin{enumerate} 
\item For any couple of probabilities $(m,m')\in \mathcal{P}_2(\mathcal{C}^d)\times \mathcal{P}_2(\mathcal{C}^d)$, for all $(t,y,y') \in [0,T] \times \shc^d \times \shc^d$,
we have
\begin{eqnarray}
\label{eq:uu'}
\vert u^m\big (t,y\big )-u^{m'}\big (t,y'\big )\vert^2\leq C_{K,\Lambda}(t)\left [ \vert y-y'\vert^2+\vert W_t(m,m')\vert ^2\right ] \ ,
\end{eqnarray}
where 
$C_{K,\Lambda}(t):=2C'_{K,\Lambda}(t)(t+2)(1+e^{2tC'_{K,\Lambda}(t)})$ with $C'_{K,\Lambda}(t)=2e^{2tM_{\Lambda}}(L^2_K+2M^2_KL^2_\Lambda t )$.
In particular the functions $C_{K, \Lambda}$ only depend on $M_K, L_K, M_\Lambda, L_\Lambda$ and $t$ and is increasing with $t$.
\item For any $(m,m') \in \shp(C^d) \times \shp(C^d)$,  for all $(t,y,y') \in [0,T] \times \shc^d \times \shc^d$, we have
\begin{eqnarray}
\label{eq:uu'2}
\vert u^m\big (t,y\big )-u^{m'}\big (t,y'\big )\vert^2\leq \mathfrak{C}_{K,\Lambda}(t)\left [ \vert y-y'\vert^2+\vert \widetilde{W}_t(m,m')\vert ^2\right ] \ ,
\end{eqnarray}
where $\mathfrak{C}_{K,\Lambda}(t) := 2e^{2tM_{\Lambda}}(\max(L_K,2M_K)^2 + 2M_{K}^2\max(L_{\Lambda},2M_{\Lambda})^2 t ) $.
 \item The function $ (m,t,x) \mapsto u^m(t,x)$ is continuous on $\shp(\shc^d) \times [0,T] \times \R^d,$ where $\shp(\shc^d)$ is endowed with the topology of weak convergence.
\item Suppose 
%to assumption~\ref{ass:main}) 
that $K \in W^{1,2}(\R^d)$. Then for any $(m,m')\in \mathcal{P}_2(\mathcal{C}^d)\times \mathcal{P}_2(\mathcal{C}^d)$, $t \in [0,T]$  
\begin{eqnarray}
\label{uu'L2}
\Vert u^m(t,\cdot)-u^{m'}(t,\cdot)\Vert_2^2 
 \leq   \tilde{C}_{K,\Lambda}(t)(1+2tC_{K,\Lambda}(t)) \vert W_t(m,m') \vert^2 \ ,
\end{eqnarray}
%More precisely, one can choose 
where 
$C_{K,\Lambda}(t):=2C'_{K,\Lambda}(t)(t+2)(1+e^{2tC'_{K,\Lambda}(t)})$ with $C'_{K,\Lambda}(t)=2e^{2tM_{\Lambda}}(L^2_K+2M^2_KL^2_\Lambda t )$ and
%Moreover, we obtain the following $L^2$ bound valid for any $t\in [0,T]$,
%\begin{equation}
%\label{eq:uu'L2}
%\Vert u^m(t,\cdot)-u^{m'}(t,\cdot)\Vert_2^2 
% \leq   \tilde{C}_{K,\Lambda}(t)(1+2tC_{K,\Lambda}(t)) \vert W_t(m,m') \vert^2 \ ,
%\end{equation}
%with $C''_{K,\Lambda}(t) := \left( C_{K,\Lambda}(t)+4L_K e^{2tM_{\Lambda}} + 2tC_{K,\Lambda}^2(t) \right)$ 
$\tilde{C}_{K,\Lambda}(t) := 2 e^{2tM_{\Lambda}}(2M_K L_{\Lambda}^2 t(t+1) 
+ \Vert \nabla K \Vert_2^2 )$, $\Vert\cdot\Vert_2$ being the standard 
$L^2(\R^d)$ or $L^2(\R^d,\R^d)$-norms. \\
% induced by the scalar product $(g,h) \in L^2(\R^d) \times L^2(\R^d) \mapsto \int_{\R^d} g(x)h(x) dx$. \\
In particular the functions $ t \mapsto C'_{K, \Lambda}(t)$ and $ t \mapsto C_{K, \Lambda}(t)$
 only depend on $M_K, L_K, M_\Lambda, L_\Lambda$  and are increasing with respect to $t$.
\item Suppose 
% to Assumption~\ref{ass:main}) 
that $\shf(K) \in L^1(\R^d)$. 
 Then there exists a constant $\bar{C}_{K,\Lambda}(t) > 0$ (depending only on $t,M_{\Lambda},L_{\Lambda},\Vert \shf(K) \Vert_1$) such that for any random probability $\eta : (\Omega,\shf) \longrightarrow (\shp_2(\shc^d),\shb(\shp(\shc^d)))$, for all $(t,m) \in [0,T] \times \shp(\shc^d)$
\begin{eqnarray}
\label{eq:uu'Linf}
  \E{ [  \Vert u^{\eta}(t,\cdot) - u^m(t,\cdot) \Vert_{\infty}^2 ] } & \leq & \bar{C}_{K,\Lambda}(t) \sup_{ \underset{\Vert \varphi \Vert_{\infty} \le 1}{\varphi \in \shc_b(\shc^d)}} 
 \E{ [ \vert \langle \eta - m , \varphi \rangle \vert^2 ]} \ ,
\end{eqnarray}
where we recall that $\shp(\shc^d)$ is endowed with the topology of weak convergence.
We remark that the expectation in both sides of \eqref{eq:uu'Linf} is taken w.r.t. the randomness of the random probability $\eta$.
\end{enumerate}
\end{prop}
\begin{rem}
\label{R26}
The map
 $\displaystyle{ d_2^{\Omega} : (\nu,\mu) \mapsto \sqrt{\sup_{ \underset{\Vert \varphi \Vert_{\infty} \le 1}{\varphi \in \shc_b(\shc^d)}} \E[ \vert \langle \nu - \mu, \varphi \rangle \vert^2 ] } }$ defines a (homogeneous) distance on $\shp^{\Omega}(\shc^d)$.
\end{rem}
The lemma  below was proved in Lemma 7.1 in \cite{LOR1}.
\begin{lem}
\label{lem:yy'}
Let $r:\,[0,T]\mapsto [0,T]$ be a non-decreasing function such that $r(s)\leq s$ for any $s\in [0,T]$ and $Y_0$ be a random variable admitting $\zeta_0$ as law. \\
Let $\shu : (t,y) \in [0,T] \times \shc^d \rightarrow \R $ (respectively $\shu' : (t,y) \in [0,T] \times \shc^d \rightarrow \R$),
 be a given Borel function  
%/ measurable w.r.t. the previsible $\sigma$-algebra on $[0,T] \times \shc^d$ ) 
such that for all $t \in [0,T]$, there is a Borel map $\shu_t: \shc ([0,t],\R^d) \rightarrow \R$ (resp. $\shu^{'}_t: \shc ([0,t],\R^d) 
\rightarrow \R$) such that $\shu(t,\cdot) = \shu_t(\cdot)$ (resp. $\shu'(t,\cdot) = \shu'_t(\cdot)$).
%$\shb_t(\shc^d) := \sigma(X_u, 0 \leq u \leq t)$ adapted. 

%(VOIR POUR CETTE HYP: and Lipschitz w.r.t. to the space variable uniformly w.r.t. to the time variable).
%Let $u$, respectively $u'$ be two $\shf_t$-adapted random fields whose trajectories are Lipschitz w.r.t. the time and space variables uniformly, almost surely. 
Then the following two assertions hold. 
\begin{enumerate}
\item Consider $Y$ (resp. $Y'$) a solution of the following SDE for $v=\shu$ (resp. $v=\shu'$):
\begin{equation}
 \label{eq:YY'}
 \begin{array}{l}
 Y_t=Y_0+\int_0^t\Phi(r(s),Y_{r(s)},v(r(s),Y_{\cdot \wedge r(s)}))dW_s+\int_0^t g(r(s),Y_{r(s)},v(r(s),Y_{\cdot \wedge r(s)}))ds\ ,\quad \textrm{for any}\ t\in [0,T] \ ,
 \end{array}
 \end{equation}
 where, we emphasize that for all $\theta \in [0,T]$, $Z_{\cdot \wedge \theta} := \{Z_u, 0 \leq u \leq \theta \} 
\in \shc([0,\theta],\R^d)$ for any continuous process $Z$. 
For any $ a \in [0,T]$, we have
\begin{equation}
 \label{eq:YY'Stab}
\E [\sup_{t\leq a} \vert Y'_t-Y_t\vert ^2]
\leq 
C_{\Phi,g}(T)\E\left [\int_0^a \vert \shu(r(t),Y_{\cdot \wedge r(t)})-\shu'(r(t),Y'_{\cdot \wedge r(t)})
\vert^2dt \right ]\ ,
\end{equation}
where $C_{\Phi,g}(T)=12(4L^2_{\Phi}+TL^2_g)e^{12T(4L^2_{\Phi}+TL^2_g)}$.
\item Suppose moreover that $\Phi$ and $g$ are $\frac{1}{2}$-H\"older w.r.t. the time and Lipschitz w.r.t. the space variables i.e. there exist some positive constants $L_{\Phi}$ and $L_g$ such that for any $(t,t',y,'y',z,z')\in [0,T]^2\times \R^{2d}\times \R^2$ 
\begin{equation}
\label{eq:PhigLipt}
\left \{
\begin{array}{l}
\vert \Phi(t,y,z)-\Phi(t',y',z')\vert \leq L_{\Phi} (\vert t-t'\vert^{\frac{1}{2}} +\vert y-y'\vert +\vert z-z'\vert)\\
\vert g(t,y,z)-g(t',y',z')\vert \leq L_{g} (\vert t-t'\vert^{\frac{1}{2}} +\vert y-y'\vert +\vert z-z'\vert)\ .
\end{array}
\right . 
\end{equation}
Let  $r_1, r_2:\,[0,T]\mapsto [0,T]$ being two non-decreasing functions verifying $r_1(s)\leq s$ and $r_2(s)\leq s$ for any $s\in [0,T]$. 
Let $Y$ (resp. $Y'$) be a solution  of~\eqref{eq:YY'} for $v=\shu$ and $r=r_1$ (resp. $v=\shu'$ and $r=r_2$). %with  $r_1, r_2:\,[0,T]\mapsto [0,T]$ being two nondecreasing functions verifying $r_1(s)\leq s$ and %$r_2(s)\leq s$ for any $s\in [0,T]$. 
Then for any $ a \in [0,T]$, the following inequality holds: 
\begin{eqnarray}
 \label{eq:YY'Stabr'}
\E [\sup_{t\leq a} \vert Y'_t-Y_t\vert ^2]
&\leq &
C_{\Phi,g}(T)\left (\Vert r_1-r_2\Vert_{L^1([0,T])} +\int_0^a\E[\vert Y'_{r_1(t)}-Y'_{r_2(t)}\vert^2] dt\right .\nonumber \\
&&
\left .+\E \left[\int_0^a\vert  \shu(r_1(t),Y_{\cdot \wedge r_1(t)})-\shu'(r_2(t),Y'_{\cdot \wedge r_2(t)})\vert^2dt´\right]\right )\ ,
\end{eqnarray}
%where for all $\theta \in [0,T]$, $Z_{\cdot \wedge \theta} := \{Z_u, 0 \leq u \leq \theta \} \in \shc^d([0,\theta],\R^d)$ for any continuous process $Z$ and 
%where $\Vert \cdot \Vert_1$ is the $L^1([0,T])$-norm.  
%where $C_{\Phi,g}(T)=12(4L^2_{\Phi}+TL^2_g)e^{12T(4L^2_{\Phi}+TL^2_g)}$. 
\end{enumerate}
\end{lem}

The theorem below was the object of Theorem 3.9 in \cite{LOR1}.
\begin{thm}
\label{prop:NSDE}
%Under Assumption~\ref{ass:main}, 
Under Assumption \ref{ass:main}, the McKean type SDE~\eqref{eq:NSDE} admits strong existence and pathwise uniqueness.
\end{thm}
For a precise formulation of the notion of existence and uniqueness for the McKean type equation  \eqref{eq:NSDE}
we refer to Definition 2.6 of \cite{LOR1}.\\
\\
We finally recall  an important non-anticipating property of the map $(m,t,x) \mapsto u^m(t,x)$, stated
in \cite{LOR1}.

% Let us fix $t \in  [0,T]$, $m_t \in \mathcal{P}(\mathcal{C}_t^d)$. 

\begin{defi} \label{D38}   Let us fix $t \in  [0,T]$.
Given a non-negative Borel measure $m$ on $(\mathcal{C}^d,\shb(\shc^d))$. 
From now on,  $m_t$ will denote the (unique) induced measure on 
$(\mathcal{C}_t^d,\shb(\shc^d_t))$ (with $\mathcal{C}_t^d := \mathcal{C}([0,t],\R^d)$) defined by
$$
\int_{\mathcal{C}_t^d} F(\phi)m_t(d\phi) = \int_{\mathcal{C}^d} F(\phi_{|_{[0,t]}})m(d\phi),
$$
where $F : \mathcal{C}_t^d \longrightarrow \R$ is bounded and continuous.
\end{defi}

\begin{rem} \label{R38}
Let $t \in  [0,T], m = \delta_{\xi} \;, \xi \in \mathcal{C}^d$. 
The induced measure $m_t$, on $\shc_t^d$, is $\delta_{(\xi_r | 0 \leq r \leq t)} $.
\end{rem}
 For each $t \in [0,T]$, the same construction as the one carried on in Theorem 3.1 in \cite{LOR1} allows us to define the unique solution to
\begin{equation}
\label{u_mt}
\begin{array}{l}
u^{m_t}(s,y) = \int_{\mathcal{C}_t^d} K(y-X_s(\omega)) 
\exp\left(\int_0^s \Lambda(r,X_r(\omega),u^{m_t}(r,X_r(\omega))) dr \right)
m_t(d\omega) \quad  \ \forall s\in [0,t]\ .
\end{array}
\end{equation}
The proposition and corollary  below were the object of Proposition 3.7 and Corollary 3.8 in \cite{LOR1}.
\begin{prop}
\label{non-anticip}
Under Assumption \ref{ass:main}, we have 
$$
\forall (s,y) \in [0,t] \times \R^d, \; u^{m}(s,y) = u^{m_t}(s,y).
$$
\end{prop}

%\begin{proof}
%By definition of $m_t$, it follows that $u^m(s,y)|_{[0,t] \times \R^d}$ is a solution of \eqref{u_mt}. Invoking the uniqueness of \eqref{u_mt} ends the proof.
%\end{proof}
\begin{corro} \label{Canticip}
Let $N \in \N$, $\xi^1, \cdots, \xi^i, \cdots, \xi^N$ be $(\mathcal{G}_t)$-adapted continuous processes, where $\mathcal{G}$ is a filtration 
(defined on some probability space) 
fulfilling the usual conditions. Let $m(d\omega) = \frac{1}{N} \sum_{i=1}^{N} \delta_{\xi^i}(d\omega)$.
Then, $(u^m(t,y))$ is a $(\mathcal{G}_t)$-adapted random field, i.e.
for any $(t,y) \in [0,T] \times \R^d$, the process is $(\mathcal{G}_t)$-adapted.

\end{corro}

\section{Particle systems approximation and propagation of chaos  }
%%%%%%%%%%%%%%%%%%%%%%%%%%%%%%%%%%%%%%%%%%%%%%
\label{SChaos}

\setcounter{equation}{0}

In this section, we introduce an interacting particle system
%focus on the propagation of chaos for an interacting particle system
$\xi = (\xi^{i,N})_{i=1,\cdots, N}$ whose empirical law will be shown to converge
to the law of the solution $Y$ of 
 the McKean type equation \eqref{eq:NSDE}. 
%when the coefficients $\Phi, g, \Lambda$ are bounded and Lipschitz. 
A consequence of the so called {\it propagation of chaos}
%We recall that the propagation of chaos consists in 
 which describes the  
asymptotic independence of the components of $\xi$ when the size $N$ of the particle system 
goes to $\infty$. That property  was introduced in
\cite{McKean} and further developed and popularized by
 \cite{sznitman}. 
The convergence of $(\xi^{i,N})_{i=1,\cdots, N}$ induces a natural approximation
 of $u$, solution of~\eqref{eq:NSDE}. \\
We suppose here the validity of Assumption \ref{ass:main}.
%For the simplicity 
%of formulation we suppose that $\Phi$ and $g$  only depend on the last variable $z$.
Let $(\Omega, \mathcal F, \P)$ be a fixed probability space,
and $(W^i)_{i=1,\cdots ,N}$ be a sequence of independent $\R^p$-valued Brownian motions. Let $(Y_0^i)_{i=1,\cdots,N}$ be i.i.d. r.v. according to $\zeta_0$. We consider $\mathbf{Y}:=(Y^i)_{i=1,\cdots ,N}$ the
 sequence of processes such that  $(Y^i, u^{m^i})$ are solutions to
\begin{equation}
\label{eq:Yi}
\left \{
\begin{array}{l}
Y^i_t=Y_0^i+\int_0^t\Phi(s,Y^i_s,u^{m^i}_s(Y^i_s))dW^i_s+\int_0^tg(s,Y^i_s,u^{m^i}_s(Y^i_s))ds \\
u^{m^i}_t(y)={\displaystyle \E\left [ K(y-Y^i_t)V_t\big (Y^i,u^{m^i}(Y^i)\big ) \right ]}\ ,\quad \textrm{with}\ m^i:=\mathcal{L}(Y^i)\ ,
\end{array}
\right .
\end{equation}
%with $m^i$ standing for the law of $Y^i$ and 
recalling that $V_t\big (Y^i,u^{m^i}(Y^i)\big )=\exp \big( \int_0^t\Lambda_s(Y^i_s,u^{m^i}_s(Y^i_s))ds \big)$. 
The existence and uniqueness of the solution of each equation is ensured by Theorem~\ref{prop:NSDE}. 
We recall that  the map                            $(m,t,y) \mapsto u^m(t,y)$
 fulfills the regularity properties given at the second and third item of Proposition \ref{lem:uu'} .

Obviously the processes  $(Y^i)_{i=1,\cdots ,N}$ are independent.
They are also identically distributed since Theorem \ref{prop:NSDE} also states uniqueness
in law.
\\
So we can define
 $m^0:=m^i$ the common distribution of the processes $Y^i, i=1,\cdots ,N$, which is of course
the law of the process $Y$, such that $(Y,u)$ is a solution of \eqref{eq:NSDE}. \\
From now on, ${\mathcal C}^{dN}$  will denote $( {\mathcal C}^{d})^N$, which is obviously isomorphic to $\shc([0,T], \R^{dN})$.
%We start observing that, 
For every $\bar{\xi} \in {\mathcal C}^{dN}$ 
we will denote 
 \begin{equation}
 \label{eq:SampleXi}
 S^N(\mathbf{\bar \xi}):=\frac{1}{N} \sum_{i=1}^N \delta_{\bar \xi^{i,N}}.
 \end{equation}
The function  $(t,x) \mapsto u^{S^N(\mathbf{\bar{\xi}})}_t(x)$ is
 obtained by composition of $m \mapsto u^m_t(x)$ (defined
in  \eqref{Eum})
 with $ m = S^N(\mathbf{\bar{\xi}})$.
% with $S^N(\mathbf{\bar \xi})$ standing for the empirical measure associated to $\mathbf{\bar \xi}:=(\bar \xi^{i,N})_{i=1,\cdots ,N}$ i.e. 

Now let us introduce  the system of equations
%the weakly interacting particle system   such that 

\begin{equation}
\label{eq:XIi}
\left \{
\begin{array}{l}
\xi^{i,N}_t= \xi^{i,N}_0 +\int_0^t\Phi(s,\xi^{i,N}_s,u^{S^N(\mathbf{\xi})}_s(\xi^{i,N}_s))dW^i_s+\int_0^tg(s,\xi^{i,N}_s,u^{S^N(\mathbf{\xi})}_s(\xi^{i,N}_s))ds  \\
\xi^{i,N}_0 = Y^i_0 \\
u^{S^N(\mathbf{\xi})}_t(y)={\displaystyle \frac{1}{N}\sum_{j=1}^N  K(y-\xi^{j,N}_t) V_t\big (\xi^{j,N},u^{S^N(\mathbf{\xi})}(\xi^{j,N})\big ) }\ .
\end{array}
\right .
\end{equation}
Conformally with  \eqref{eq:SampleXi}, we consider  the empirical (random) 
measure  
 $\displaystyle{ S^N(\mathbf{Y})=\frac{1}{N} \sum_{i=1}^N \delta_{Y^{i}} }$ 
related to 
${\bf Y} := (Y^i)_{i=1,\cdots,N},$
where we recall that for each $i \in \{1,\cdots,N\}$, $Y^i$ is solution of \eqref{eq:Yi}.
 We observe that by Remark \ref{RDelta}, $S^N(\mathbf{\xi})$ and $S^N(\mathbf{Y})$ are measurable maps from $(\Omega,\shf)$ to $(\shp(\shc^d),\shb(\shp(\shc^d)))$; moreover $S^N(\xi),S^N({\bf{Y}}) \in \shp_2(\shc^d)$ $\P$-a.s.
A solution ${\bf{\xi}}:=(\xi^{i,N})_{i=1,\cdots ,N}$  of \eqref{eq:XIi} is called
{\bf{interacting particle system}}.

The first line of \eqref{eq:XIi} is in fact a path-dependent stochastic differential equation.
We claim that its coefficients 
%of the first line equation of \eqref{eq:XIi}
 are measurable. Indeed, the map $(t,\bar{\xi}) \mapsto (S^{N}(\bar{\xi}),t,\bar{\xi}_t^i,)$ being continuous from $([0,T] \times \shc^{dN},\shb([0,T]) \otimes \shb(\shc^{dN}))$ to $( \shp(\shc^{d}) \times [0,T] \times \R^d , \shb(\shp(\shc^d)) \otimes \shb([0,T]) \otimes \shb(\R^d)) $ for all $i \in \{1,\cdots,N\}$, by composition with the continuous map $(m,t,y) \mapsto u^m(t,y)$ (see Proposition \ref{lem:uu'} 3.) we deduce the continuity of $(t,\bar{\xi}) \mapsto (u_t^{S^N(\bar{\xi})}(\bar{\xi}^i_t))_{i=1,\cdots,N}$, and so the measurability from $([0,T] \times \shc^{dN},\shb([0,T]) \otimes \shb(\shc^{dN}))$ to $(\R,\shb(\R))$.
\\
In the sequel, for simplicity we set $\bar{\xi}_{r \le s}:= (\bar{\xi}^i_{r \le s})_{1 \le i \le N}$.  
We remark that,  by Proposition \ref{non-anticip} and Remark \ref{R38}, we have
\begin{eqnarray}
\label{NonAntip}
 \left (u^{S^N(\bar{\xi})}_s(\bar{\xi}^{i}_s)\right )_{i=1,\cdots N} = \left (u^{S^N(\mathbf{\bar{\xi}_{r \le s} })}_s(\bar{\xi}^{i}_s)\right )_{i=1,\cdots N},
\end{eqnarray}
for any $s \in [0,T], \; \bar{\xi} \in {\mathcal C}^{dN}$ and so stochastic integrands of \eqref{eq:XIi} 
are adapted (so progressively measurable being continuous in time) and so the corresponding
 It\^o integral makes sense. We discuss below the well-posedness of
\eqref{eq:XIi}. \\
The fact that \eqref{eq:XIi} has a unique (strong) solution $(\xi^{i,N})_{i=1,\cdots N}$ holds true because of the following arguments.
\begin{enumerate}
%Indeed the diffusion and drift coefficients in the first equation of~\eqref{eq:XIi} satisfy the standard conditions since 
 \item  $\Phi$ and $g$ are Lipschitz. Moreover 
 the map $\mathbf{\bar{\xi}}_{r\leq s}
\mapsto \left (u^{S^N(\mathbf{\bar{\xi}_{r \le s} })}_s(\bar{\xi}^{i}_s) \right)_{i=1,\cdots,N}$
is Lipschitz. \\
Indeed, for given $(\xi_{r \leq s},\eta_{r \leq s}) \in \shc^{dN} \times \shc^{dN} $, $s \in [0,T]$, by using successively inequality \eqref{eq:uu'} of Proposition \ref{lem:uu'} and Remark \ref{RADelta}, for all $i \in \{1,\cdots,N\}$ we have 
\begin{eqnarray}
\label{Lip_u_PD}
\vert u_s^{S^N(\xi_{r \leq s})}(\xi^i_{t}) - u_s^{S^N(\eta_{r \leq s})}(\eta^i_{t})  \vert & \leq & \sqrt{C_{K,\Lambda}(T)} \left(  \vert \xi^i_s - \eta_s^i \vert + \frac{1}{N} \sum_{j=1}^N \sup_{0 \leq r \leq s}\vert \xi^j_r- \eta^j_r \vert \right) \nonumber \\
& \leq & 2\sqrt{C_{K,\Lambda}(T)} \max_{j=1,\cdots,N} \sup_{0 \leq r \leq s}\vert \xi^j_r-\eta^j_r \vert \ .
\end{eqnarray}
Finally  the functions 
\begin{eqnarray*}
 {\mathbf{\bar{\xi}}}_{r\leq s}  &\mapsto& 
\left (\Phi(s, \bar \xi^i_s,u_s^{S^N(\mathbf{\bar{\xi}_r, r \le s})}
(\bar{\xi}^{i}_s))\right )_{i=1,\cdots N} \\
 \mathbf{\bar{\xi}}_{r\leq s}  &\mapsto& \left (g(s, \bar \xi^i_s, u_s^{S^N
(\mathbf{\bar{\xi}_r, r \le s})}
(\bar{\xi}^{i}_s))\right )_{i=1,\cdots N}
\end{eqnarray*}
are uniformly Lipschitz and bounded.
\item 
A classical argument of well-posedness for systems
of path-dependent stochastic differential equations
with Lipschitz dependence on the sup-norm of the path, see Chapter~V, Section~2.11, Theorem~11.2 page 128 in~\cite{rogers_v2}.
\end{enumerate}
After the preceding introductory considerations, we can state and prove the main theorem of the section.
\begin{thm} \label{TPC}
Let us suppose the validity of Assumption \ref{ass:main}. Let $N$ be a fixed positive integer.
Let $(Y^i)_{i=1, \cdots,N}$ (resp. ($(\xi^{i,N})_{i=1, \cdots,N}$) be the solution of \eqref{eq:Yi} (resp. \eqref{eq:XIi}), 
let $m^0$ as defined after \eqref{eq:Yi}. The following assertions hold.
\begin{enumerate}
\item If $\shf(K) \in L^1(\R^d)$, there is a positive constant $C$  only depending on 
$L_\Phi, L_g,M_K,M_{\Lambda},L_K,L_{\Lambda},T, \Vert \shf(K) \Vert_1  $,
% $C= C(\Phi, g, \Lambda, K, T)$ 
such that, for all $i=1,\cdots,N$ and $t \in [0,T]$,  
\begin{eqnarray}
\label{eq:xiYuFinalb1}
 \E[ \Vert u_t^{S^N(\xi)}-u^{m^0}_t\Vert_{\infty} ^2] 
 & \leq & \frac{C}{N} \\
  \label{eq:xiYuFinalb2}
\E[\sup_{0 \leq s\leq t} \vert \xi^{i,N}_s-Y^i_s\vert ^2] & \leq & \frac{C}{N}\ .
\end{eqnarray}

\item  If $K$ belongs to $W^{1,2}(\R^d)$, there is a positive constant
  $C$  only depending on $L_\Phi, L_g, M_K,M_{\Lambda},L_K,L_{\Lambda},T$ and $\Vert \nabla K\Vert_2$,
% $C= C(\Phi, g, \Lambda, K, T)$
 such that, for all $t \in [0,T]$,
 \begin{eqnarray}
 \label{eq:xiYuFinal3}
 \E[ \Vert u_t^{S^N(\xi)}-u^{m^0}_t\Vert_{2} ^2] & \leq & \frac{C}{N} \ .
 \end{eqnarray}
\end{enumerate} 
\end{thm}
Before proving Theorem \ref{TPC}, we remark that the propagation of chaos follows easily.
\begin{corro}
\label{CCP}
Under Assumption \ref{ass:main}, the propagation of chaos holds for the interacting particle system $(\xi^{i,N})_{i \in \N}$.
\end{corro}
\begin{proof}
We prove here that Theorem \ref{TPC} implies the propagation of chaos. \\ 
Indeed, for all $k \in \N^{\star}$, \eqref{eq:xiYuFinalb2} implies 
$$
(\xi^{1,N}-Y^1,\xi^{2,N}-Y^2,\cdots,\xi^{k,N}-Y^k) \xrightarrow[\text{$N \longrightarrow + \infty$}]{\text{$L^{2}(\Omega,\shf,\P)$}} 0 \ ,
$$
which implies in particular the convergence in law of the vector $(\xi^{1,N},\xi^{2,N},\cdots,\xi^{k,N})$ to $(Y^1,Y^2,\cdots,Y^k)$. Consequently,
 since $({Y}^{i})_{i=1,\cdots,k}$ are i.i.d. according to $m^0$
\begin{eqnarray}
\label{eq:CVChaos}
(\xi^{1,N},\xi^{2,N},\cdots,\xi^{k,N}) \; \textrm{ converges in law to } (m^0)^{\otimes k} \; \textrm{ when } N \rightarrow +\infty \ .
\end{eqnarray}
\end{proof}

The validity of \eqref{eq:xiYuFinalb1} and \eqref{eq:xiYuFinalb2}
 will be the consequence of the significant more general proposition below.
\begin{prop}\label{PMoreGeneral}
Let us suppose the validity of Assumption \ref{ass:main}. Let $N$ be a fixed positive integer. Let $(W^{i,N})_{i=1,\cdots,N}$ be a family of $p$-dimensional standard Brownian motions (not necessarily independent).
Let  $({Y}^i_0)_{i=1,\cdots,N}$ be the family of i.i.d. r.v. initializing the system~\eqref{eq:Yi}.
We consider the processes $(\bar{Y}^{i,N})_{i=1,\cdots,N}$, such that 
for each $i \in \{1,\cdots,N\}$, $\bar Y^{i,N}$ is the unique strong solution of
\begin{equation}
\label{eq:Yi-general}
\left \{
\begin{array}{l}
\bar{Y}^{i,N}_t=Y_0^i+\int_0^t\Phi(s,\bar{Y}^{i,N}_s,u^{m^{i,N}}_s(\bar{Y}^{i,N}_s))dW^{i,N}_s+\int_0^tg(s,\bar{Y}^{i,N}_s,u^{m^{i,N}}_s(\bar{Y}^{i,N}_s))ds, \quad \textrm{ for all } t \in [0,T] \\
u^{m^{i,N}}_t(y)={\displaystyle \E\left [ K(y-\bar{Y}^{i,N}_t)V_t\big (\bar{Y}^{i,N},u^{m^{i,N}}(\bar{Y}^{i,N})\big ) \right ]}\ ,\quad \textrm{with}\ m^{i,N}:=\mathcal{L}(\bar{Y}^{i,N})\ ,
\end{array}
\right .
\end{equation}
recalling that $V_t\big (Y^{i,N},u^{m^{i,N}}(Y^{i,N})\big )=
\exp \big( \int_0^t\Lambda(s,Y^{i,N}_s,u^{m^{i,N}}_s(Y^{i,N}_s))ds \big)$. \\
Let us consider now the system of equations
% obtained by replacing $W^i$ by
% $W^{i,N}$ in~
\eqref{eq:XIi},
where the processes  $W^i$ are replaced  by
 $W^{i,N}$, i.e.
\begin{equation}
\label{eq:XIiN}
\left \{
\begin{array}{l}
\xi^{i,N}_t= \xi^{i,N}_0 +\int_0^t\Phi(s,\xi^{i,N}_s,u^{S^N(\mathbf{\xi})}_s(\xi^{i,N}_s))dW^{i,N}_s+\int_0^tg(s,\xi^{i,N}_s,u^{S^N(\mathbf{\xi})}_s(\xi^{i,N}_s))ds  \\
\xi^{i,N}_0 = Y^{i}_0 \\
u^{S^N(\mathbf{\xi})}_t(y)={\displaystyle \frac{1}{N}\sum_{j=1}^N  K(y-\xi^{j,N}_t) V_t\big (\xi^{j,N},u^{S^N(\mathbf{\xi})}(\xi^{j,N})\big ) }\ .
\end{array}
\right .
\end{equation}
Then the following assertions hold.
\begin{enumerate}
\item For any $i=1,\cdots N$, $(\bar{Y}^{i,N}_t)_{t\in [0,T]}$ have the same law $m^{i,N}=m^0$, where $m^0$ is the common law of processes $(Y^i)_{i=1,\cdots,N}$ defined by the system \eqref{eq:Yi}.
% the underlying law of $(Y^i)$ solution of~\eqref{eq:Yi}.
\item Equation~\eqref{eq:XIiN} admits a unique strong solution.
	\item Suppose moreover that $\shf(K)$ is in $L^1(\R^d)$. Then  there is a positive constant $C$  only depending on 
$L_\Phi, L_g,M_K,M_{\Lambda},L_K,L_{\Lambda},T$ and $ \Vert \shf(K) \Vert_1 $
%there is a constant $C= C(K,\Phi, g, \Lambda, T)$ 
such that, for all $t \in [0,T]$,
\begin{equation}
\label{eq:xiYuFinalGeneral}
\sup_{i=1,\dots,N} \E[\sup_{0 \leq s\leq t} \vert \xi^{i,N}_s-\bar{Y}^{i,N}_s\vert ^2] + \E[ \Vert u_t^{S^N(\xi)}-u^{m^0}_t\Vert_{\infty} ^2] 
  \leq  C\sup_{ \underset{\Vert \varphi \Vert_{\infty} \le 1}{\varphi \in \shc_b(\shc^d)}} \E[ \vert \langle S^N(\mathbf{\bar{Y}})  - m^0 , \varphi \rangle \vert^2 ], 
\end{equation}
with again $\displaystyle{S^{N}({\bf{\bar{Y}}}) := \frac{1}{N} \sum_{j=1}^N \delta_{\bar{Y}^{j,N}}}$. 
\end{enumerate}
\end{prop}
\begin{rem} \label{RMoreGeneral}
The convergence of the numerical approximation $u_t^{S^N(\xi)}$ to $u_t^{m^0}$ only requires the convergence of
 $ d_2^\Omega(S^N(\mathbf{\bar Y}), m^0)$ to $0$, where  the distance $ d_2^\Omega$ has been defined at Remark \ref{R26}.
This holds if, for each $N$, $\bar Y^{i,N}, i=1,\cdots N$ are independent; however, this is only a sufficient condition.\\  
This gives the opportunity to define new numerical schemes for which the convergence of the empirical measure $ S^N(\mathbf{\bar{Y}})$ is verified without i.i.d. particles. Let us consider $(\bar{Y}^{i,N})_{i=1,\cdots N}$ (resp. $(\xi^{i,N})_{i=1,\cdots N}$) solutions of~\eqref{eq:Yi-general} (resp.~\eqref{eq:XIiN}). Observe that for any real valued test function in $\mathcal{C}_b(\mathcal{C}^d)$
 $$
 \E[\langle S^N(\mathbf{\bar{Y}})-m^0 ,\varphi \rangle ^2]=\frac{\sigma^2_{\varphi}}{N}(1+\frac{2}{N} \sum_{i<j} \rho^{i,j}_{\varphi})\ ,
 $$
 where $\sigma_{\varphi}:=\sqrt{Var ( \varphi(\bar Y^{1,N}))}$ and $\rho^{i,j}_{\varphi}:= \frac{\E[\varphi(Y^{i,N})\varphi(Y^{j,N})]-\E[\varphi(Y^{i,N})]\E[\varphi(Y^{j,N})]}{\sigma^2_{\varphi}}$. \\
 In the specific case where $(W^{i,N})_{i=1,\cdots N}$ are independent Brownian motions then  $\rho^{i,j}_{\varphi}=0$ for any bounded $\varphi \in \mathcal{C}_b(\mathcal{C}^d)$ and 
 \begin{eqnarray}
 \label{majorNormFaible}
 \sup_{ \underset{\Vert \varphi \Vert_{\infty} \le 1}{\varphi \in \shc_b(\shc^d)}} \E[ \vert \langle S^N(\mathbf{\bar{Y}})  - m^0 , \varphi \rangle \vert^2 ] \leq \frac{1}{N} \ .
 \end{eqnarray}
 %and we derive the convergence of $u^{S^N(\xi)}_t$ to $u^{m^0}_t$ at a rate $1/\sqrt{N}$.
 With our error bound one can naturally  investigate antithetic variables approaches to improve the interacting particle system convergence. Let us consider $N=2N'$ and take
 $(W^{i,N})_{i=1,\cdots N'}$ as $N'$ iid Brownian motions, then for the rest of the particles, for any $j=N'+1,N'+2, \cdots 2N'$, set $W^{j,N}=-W^{j-N',N}$.
 In this situation, we obtain
 $$
 \E[\langle S^N(\mathbf{\bar{Y}})-m^0 ,\varphi \rangle ^2]=\frac{\sigma^2_{\varphi}}{N}(1+\rho^{1,1+N'}_{\varphi}).
 $$
 So, even in this case, the rate of convergence of $u^{S^N(\xi)}_t$ to $u^{m^0}_t$ 
is still of order
 $1/\sqrt{N}$.  \\ If moreover one has
 $\displaystyle{ \sup_{ \underset{\Vert \varphi \Vert_{\infty} \le 1}{\varphi \in \shc_b(\shc^d)}} \rho^{1,1+N'}_{\varphi} \leq 0 }$,
 the variance will also be reduced with respect to the case of independent
Brownian motions, see  \eqref{majorNormFaible}.

\end{rem}
\begin{proof}[Proof of Proposition \ref{PMoreGeneral}]  Let us fix $t \in [0,T]$. 
In this proof, $C:= C(\Phi, g, \Lambda, K, T)$
 is a real positive constant, 
 which may change from line to line. \\
  Equation~\eqref{eq:Yi-general} has $N$ blocks, numbered by $1 \le i \le N$.
Theorem~\ref{prop:NSDE} gives 
uniqueness in law for each block equation,
  which implies that for any $i=1,\cdots N$, $m^{i,N}=m^0$ and proves the
 first item. \\
 Concerning  item 2., i.e. the strong existence and pathwise uniqueness of \eqref{eq:XIiN}, 
the same argument as for the well-statement of \eqref{eq:XIi} operates. 
 The only difference consists 
in the fact that the Brownian motions may be correlated.
A very close proof to the one of
 Theorem 11.2 page 128 in
 \cite{rogers_v2} works: the main argument is 
  the multidimensional BDG inequality, see e.g. Problem 3.29 of  
\cite{karatshreve}. \\
%We prove first the inequalities \eqref{eq:xiYuFinalb1} and \eqref{eq:xiYuFinalb2}.
We discuss now item 3. proving inequality \eqref{eq:xiYuFinalGeneral}. 
On the one hand, since the map $(t,\bar{\xi}) \in [0,T] \times \shc^{dN} \mapsto~(u_t^{S^N(\bar{\xi})}(\bar{\xi}^i_t))_{i=1,\cdots,N}$ is measurable and satisfies the non-anticipative property \eqref{NonAntip}, the first assertion of Lemma \ref{lem:yy'} gives for all $i \in \{1,\cdots,N\}$
\begin{eqnarray}
\label{PropaChaos}
%\label{PC78}
\E[\sup_{0 \leq s\leq t} \vert \xi^{i,N}_s-\bar{Y}^{i,N}_s\vert ^2]
& \leq & C
\E[\int_0^t
\vert u_s^{S^N(\mathbf{\xi})}(\xi^{i,N}_s)-u^{m^0}_s(\bar{Y}^{i,N}_s)\vert ^2
ds ] \nonumber \\
& \leq & C
\int_0^t \E[
\vert u_s^{S^N(\mathbf{\xi})}(\xi^{i,N}_s)-u^{m^0}_s(\xi^{i,N}_s)\vert ^2]
ds + \int_0^t
\E[ \vert u^{m^0}_s(\xi^{i,N}_s)-u^{m^0}_s(\bar{Y}^{i,N}_s)\vert ^2]
ds \nonumber \\
& \leq & C \int_0^t \left( \E[ \Vert u_s^{S^N(\mathbf{\xi})}-u^{m^0}_s \Vert_{\infty}^2 ] + 
\E[ \sup_{0 \leq r \leq s} \vert \xi^{i,N}_r - \bar{Y}^{i,N}_r \vert^2 ] \right) ds,\quad \textrm{ by \eqref{eq:uu'} ,} \nonumber \\
\end{eqnarray}
which implies
\begin{eqnarray}
\label{um0um0}
\sup_{i=1,\cdots,N} \E[\sup_{0 \leq s\leq t} \vert \xi^{i,N}_s-\bar{Y}^{i,N}_s\vert ^2] \leq C \int_0^t \left( \E[ \Vert u_s^{S^N(\mathbf{\xi})}-u^{m^0}_s \Vert_{\infty}^2 ] + \sup_{i=1,\cdots,N} \E[ \sup_{0 \leq r \leq s} \vert \xi^{i,N}_r - \bar{Y}^{i,N}_r \vert^2 ] \right) ds.
\end{eqnarray}
We use inequalities \eqref{eq:uu'} for 
%(applied pathwise with
 $m = S^N(\xi)(\bar{\omega})$ and $m' = S^N({\bf{\bar{Y}}})(\bar{\omega})$),
where $\bar{\omega}$ is a random realization in $\Omega$
 and \eqref{eq:uu'Linf} (with the random probability $ \eta = S^N({\bf{\bar{Y}}})$ and $m = m^0$)
 in item 5. of Proposition \ref{lem:uu'}. This  yields
\begin{eqnarray}
\label{uSNxium0}
\E[ \Vert u_t^{S^N(\mathbf{\xi})}-u^{m^0}_t \Vert_{\infty}^2 ] & \leq & 2 
\E\left[ \Vert u_t^{S^N(\mathbf{\xi})}-u^{S^N(\mathbf{\bar{Y}})}_t \Vert_{\infty}^2 \right]
+ 2  \E[ \Vert u^{S^N(\mathbf{\bar{Y}})}_t-u_t^{m^0} \Vert_{\infty}^2 ] \nonumber \\
& \leq & 2 C \E[ \vert W_t(S^N(\mathbf{\xi}),S^N(\mathbf{\bar{Y}})) \vert^2] + 2 C \sup_{ \underset{\Vert \varphi \Vert_{\infty} \le 1}{\varphi \in \shc_b(\shc^d)}} \E{ [ \vert \langle S^N(\mathbf{\bar{Y}})  - m^0 , \varphi \rangle \vert^2 ]} \nonumber \\
& \leq & \frac{2 C}{N} \sum_{i=1}^N \E[\sup_{0 \leq s\leq t} \vert \xi^{i,N}_s-\bar{Y}^{i,N}_s \vert ^2] + C \sup_{ \underset{\Vert \varphi \Vert_{\infty} \le 1}{\varphi \in \shc_b(\shc^d)}} \E[ \vert \langle S^N(\mathbf{\bar{Y}})  - m^0 , \varphi \rangle \vert^2 ] \nonumber \\
& \leq & 2 C \sup_{i=1,\cdots,N} \E[\sup_{0 \leq s\leq t} \vert \xi^{i,N}_s-\bar{Y}^{i,N}_s \vert ^2] \nonumber \\
&& + \; C \sup_{ \underset{\Vert \varphi \Vert_{\infty} \le 1}{\varphi \in \shc_b(\shc^d)}} \E[ \vert \langle S^N(\mathbf{\bar{Y}})  - m^0 , \varphi \rangle \vert^2 ],
\end{eqnarray}
where the third inequality follows from Remark \ref{RADelta}. \\
Let us introduce the non-negative function $G$ defined on $[0,T]$ by 
%$$g(t):= \E[ \Vert u_t^{S^N(\mathbf{\xi})}-u^{m^0}_t \Vert_{\infty}^2 ] + \frac{1}{N} \sum_{j=1}^N \E[\sup_{0 \leq s\leq t} \vert \xi^{j,N}_{s}-Y^j_{s} \vert ^2] \ .
%$$
$$
G(t):= \E[ \Vert u_t^{S^N(\mathbf{\xi})}-u^{m^0}_t \Vert_{\infty}^2 ] + \sup_{i=1,\cdots,N} \E[\sup_{0 \leq s\leq t} \vert \xi^{i,N}_s-\bar{Y}^{i,N}_s \vert ^2] \ .
$$
From  inequalities \eqref{um0um0} and \eqref{uSNxium0} that are valid for all $t \in [0,T]$, we obtain
\begin{eqnarray}
G(t) & \leq & % 3 OLD
(2C + 1) \sup_{i=1,\cdots,N} \E[\sup_{0 \leq s\leq t} \vert \xi^{i,N}_s-\bar{Y}^{i,N}_s \vert ^2] + C \sup_{ \underset{\Vert \varphi \Vert_{\infty} \le 1}{\varphi \in \shc_b(\shc^d)}} \E[ \vert \langle S^N(\mathbf{\bar{Y}})  - m^0 , \varphi \rangle \vert^2 ] \nonumber \\
& \leq & C \int_0^t \left( \E[ \Vert u_s^{S^N(\mathbf{\xi})}-u^{m^0}_s \Vert_{\infty}^2 ] + \sup_{i=1,\cdots,N} \E[ \sup_{0 \leq r \leq s} \vert \xi^{i,N}_r - \bar{Y}^{i,N}_r \vert^2 ] \right) ds \nonumber \\
& & + \; C \sup_{ \underset{\Vert \varphi \Vert_{\infty} \le 1}{\varphi \in \shc_b(\shc^d)}} \E[ \vert \langle S^N(\mathbf{\bar{Y}})  - m^0 , \varphi \rangle \vert^2 ] \nonumber \\
& \leq & C \int_0^t G(s) ds + C \sup_{ \underset{\Vert \varphi \Vert_{\infty} \le 1}{\varphi \in \shc_b(\shc^d)}} \E[ \vert \langle S^N(\mathbf{\bar{Y}})  - m^0 , \varphi \rangle \vert^2 ] \ .
\end{eqnarray}
By Gronwall's lemma, for all $t \in [0,T]$, we obtain
\begin{eqnarray}
\label{PCh1}
\E[ \Vert u_t^{S^N(\mathbf{\xi})}-u^{m^0}_t \Vert_{\infty}^2 ] + \sup_{i=1,\cdots,N} \E[\sup_{0 \leq s\leq t} \vert \xi^{i,N}_{s}-\bar{Y}^{i,N}_{s} \vert ^2] \leq Ce^{Ct} \sup_{ \underset{\Vert \varphi \Vert_{\infty} \le 1}{\varphi \in \shc_b(\shc^d)}} \E[ \vert \langle S^N(\mathbf{\bar{Y}})  - m^0 , \varphi \rangle \vert^2 ] \ .
\end{eqnarray}
This concludes the proof of Proposition \ref{PMoreGeneral}.
\end{proof}
From now on, we prove Theorem \ref{TPC},
\begin{proof}[Proof of Theorem \ref{TPC}]
As we have mentioned above we will apply 
%Inequalities \eqref{eq:xiYuFinalb1} and \eqref{eq:xiYuFinalb2} will be a consequence of
Proposition \ref{PMoreGeneral} setting for all $i \in \{1,\cdots,N\}$, $W^{i,N} := W^i$.
 Pathwise uniqueness of systems \eqref{eq:Yi} and \eqref{eq:Yi-general} implies $\bar{Y}^{i,N} = Y^i$ for all $i \in \{1,\cdots,N\}$.
Taking into account \eqref{eq:xiYuFinalGeneral} in
Proposition \ref{PMoreGeneral}, in order to establish
inequalities \eqref{eq:xiYuFinalb1} and \eqref{eq:xiYuFinalb2},
 we need to bound the quantity $\displaystyle{\sup_{ \underset{\Vert \varphi \Vert_{\infty} \le 1}{\varphi \in \shc_b(\shc^d)}} \E[ \vert \langle S^N(\mathbf{Y})  - m^0 , \varphi \rangle \vert^2 ] } $ .
% To this end, it is enough to apply Proposition \ref{PMoreGeneral},
% in particular \eqref{eq:xiYuFinalGeneral},
%  by setting for all $i \in \{1,\cdots,N\}$, $W^{i,N} := W^i$. Pathwise uniqueness of systems \eqref{eq:Yi} and \eqref{eq:Yi-general} implies $\bar{Y}^{i,N} = Y^i$ for all $i \in \{1,\cdots,N\}$.
This is possible via \eqref{majorNormFaible} in Remark \ref{RMoreGeneral},
since $(Y^i)_{i=1,\cdots,N}$ are i.i.d. according to $m^0$.
 %inequalities \eqref{eq:xiYuFinalb1} and \eqref{eq:xiYuFinalb2} follow from item 1. of Remark \ref{RMoreGeneral}. \\
This concludes the proof of item 1.\\
It remains now to prove \eqref{eq:xiYuFinal3} in item 2. First, the  inequality  
\begin{equation}
\label{SNxim0L2}
\E[ \Vert u_t^{S^N(\xi)}-u^{m^0}_t\Vert_{2} ^2] \leq 2\E[ \Vert u_t^{S^N(\xi)}-u^{S^N(\mathbf{Y})}_t\Vert_{2} ^2] + 2\E[ \Vert u_t^{S^N(\mathbf{Y})}-u^{m^0}_t\Vert_{2} ^2],
\end{equation}
holds for all $t \in [0,T]$.
Using  inequality \eqref{uu'L2} of Proposition \ref{lem:uu'}, for all $t \in [0,T]$,
for $m = S^N(\xi), m'= S^N(\mathbf{Y})$,
we get
\begin{eqnarray}
\label{SNxiSNYL2}
\E[ \Vert u_t^{S^N(\xi)}-u^{S^N(\mathbf{Y})}_t\Vert_{2} ^2] & \leq & C \E{[ W_t(S^N(\xi),S^N(\mathbf{Y}))^2]} \nonumber \\
& \leq & C \frac{1}{N} \sum_{j=1}^N \E{[ \sup_{0 \leq r \leq t} \vert \xi^{j,N}_r-Y^j_r \vert^2 ]} \nonumber \\
& \leq & \frac{C}{N} \ ,
\end{eqnarray}
where the latter inequality is obtained through \eqref{eq:xiYuFinalb2}. The second term of the r.h.s. in \eqref{SNxim0L2} needs more computations. Let us fix $i \in  \{1,\cdots,N\}$.
First, 
\begin{equation}
\label{major-uu'L2}
\E[ \Vert u_t^{S^N(\mathbf{Y})}-u^{m^0}_t\Vert_{2} ^2] \leq 2 \left( \E{[ \Vert A_t \Vert_2^2 ]} + \E{[ \Vert B_t \Vert_2^2 ]}  \right) \ ,
\end{equation}
where,
for all $t \in [0,T]$
\begin{equation}
\label{def-AB}
\left \{
\begin{array}{l}
A_t(x):= {\displaystyle \frac{1}{N}\sum_{j=1}^N K(x-Y^j_t)\Big [ V_t\big (Y^j,u^{S^N(\mathbf{Y})}(Y^j)\big ) - V_t\big (Y^j,u^{m^0}(Y^j)\big )\Big ]} \\
B_t(x) := \displaystyle{ \frac{1}{N} \sum_{j=1}^N  K(x-Y^j_t) V_t\big (Y^j,u^{m^0}(Y^j)\big )-\E\Big [K(x-Y^1_t) V_t\big (Y^1,u^{m^0}(Y^1)\big ) \Big]} \ ,
\end{array}
\right .
\end{equation}
where we recall that $m^0$ is the common law of all the processes $Y^i, 1 \le i \le N$. \\
To simplify notations, we set $P_j(t,x) := K(x-Y^j_t) V_t\big (Y^j,u^{m^0}(Y^j)\big )-\E\Big [K(x-Y^1_t) V_t\big (Y^1,u^{m^0}(Y^1)\big ) \Big]$ for all $j \in \{1,\cdots,N\}$, $x \in \R^d$ and $t \in [0,T]$. \\
We observe that for all $x \in \R^d, t \in [0,T]$, $(P_j(t,x))_{j=1,\cdots,N}$ are i.i.d. centered r.v.  Hence, 
$$
\E{[B_t(x)^2]} = \frac{1}{N} \E{[P_1^2(t,x)]} = \frac{1}{N}Var \big(P_1(t,x) \big) \leq \frac{1}{N} \E{[K^2 (x-Y^1_t)
 V^2_t\big (Y^1,u^{m^0}(Y^1)\big )]} \leq \frac{M_K e^{2tM_{\Lambda}}}{N} \E{[ K(x-Y^1_t) ]}.
$$ 
By integrating each side of the  inequality above w.r.t. $x \in \R^d$, we obtain
\begin{equation}
\label{majorA}
 \E{\left[ \int_{\R^d} \vert B_t(x) \vert^2 dx\right]} = \int_{\R^d} \E{[\vert B_t(x) \vert^2]} dx \leq \frac{M_K e^{2tM_{\Lambda}}}{N} \ ,
\end{equation}
where we have used that $\Vert K \Vert_1 = 1$. \\
Concerning $A_t(x)$, we write
\begin{eqnarray}
\label{majorB}
\vert A_t(x) \vert^2 & \leq & \frac{1}{N}\sum_{j=1}^N K(x-Y^j_t)^2\Big [ V_t\big (Y^j,u^{S^N(\mathbf{Y})}(Y^j)\big )
- V_t\big (Y^j,u^{m^0}(Y^j)\big )\Big ]^2 \nonumber \\
& = & \frac{1}{N}\sum_{j=1}^N K(x-Y^j_t) K(x-Y^j_t)\Big [ V_t\big (Y^j,u^{S^N(\mathbf{Y})}(Y^j)\big )
- V_t\big (Y^j,u^{m^0}(Y^j)\big )\Big ]^2 \nonumber \\
& \leq & \frac{M_K T}{N} e^{2tM_{\Lambda}}L_{\Lambda}^2 \sum_{j=1}^N K(x-Y^j_t) \int_0^t \vert u_r^{S^N(\mathbf{Y})}(Y^j_r)-u_r^{m^0}(Y^j_r) \vert^2 dr \nonumber \\
% & \leq & \frac{M_K T}{N} e^{2tM_{\Lambda}}L_{\Lambda}^2 \sum_{j=1}^N K(x-Y^j_t) \int_0^t \sup_{x \in %\R^d}\vert u_r^{S^N(\mathbf{Y})}(x)-u_r^{m^0}(x) \vert^2 dr \nonumber \\
& \le & \frac{M_K T}{N} e^{2tM_{\Lambda}}L_{\Lambda}^2 \sum_{j=1}^N K(x-Y^j_t) \int_0^t \Vert u_r^{S^N(\mathbf{Y})}-u_r^{m^0} \Vert_{\infty}^2 dr,  \nonumber \\
\end{eqnarray}
where the third inequality comes from \eqref{eq:Vmajor2}. Integrating w.r.t. $x \in \R^d$ and taking expectation on each side of the above inequality gives us,
 for all $t \in [0,T]$,
\begin{eqnarray}
\label{majorBFinal}
\E{[ \int_{\R^d} \vert A_t(x) \vert^2 dx]} & \leq & M_K T e^{2tM_{\Lambda}}L_{\Lambda}^2 \int_0^t \E{[ \Vert u_r^{S^N(\mathbf{Y})}-u_r^{m^0} \Vert_{\infty}^2 ]} dr \nonumber \\
& \leq & M_K T^2 e^{2tM_{\Lambda}}L_{\Lambda}^2 C \sup_{ \underset{\Vert \varphi \Vert_{\infty} \le 1}{\varphi \in \shc_b(\shc^d)}} 
  \E{ [ \vert \langle S^N(\mathbf{Y})  - m^0 , \varphi \rangle \vert^2 ]} \nonumber \\
  & \leq & \frac{M_K T^2 e^{2tM_{\Lambda}}L_{\Lambda}^2 C}{N} \ ,
\end{eqnarray}
where we have used \eqref{eq:uu'Linf} of Proposition \ref{lem:uu'} for the second inequality above and \eqref{majorNormFaible} for the latter one. To conclude, it is enough to replace \eqref{majorA}, \eqref{majorBFinal} in \eqref{major-uu'L2}, and inject \eqref{SNxiSNYL2}, \eqref{major-uu'L2} into \eqref{SNxim0L2}.

\end{proof}

\section{Particle algorithm }
%%%%%%%%%%%%%%%%%%%%%%%%%%%%%%%%%%%
\label{S8}

\setcounter{equation}{0}

\subsection{Time discretization of the particle system}
\label{time-discretization} 
%%%%%%%%%%%%%%%%%%%%%%%%%%%%%%%%%%%

In this section Assumption \ref{ass:main2}. is in force.
Let $(Y_0^i)_{i=1,\cdots,N}$ be i.i.d. r.v. distributed according to $\zeta_0$.
In the sequel, we are interested in discretizing the interacting particle system
\eqref{eq:XIi}.   $(\xi^{i,N}, 1 \le i \le N$)  will denote again the corresponding solution.
Let us consider a regular time grid $0=t_0\leq \cdots\leq t_k=k\delta t\leq \cdots \leq t_n=T$, with $\delta t=T/n$.
We introduce the  continuous $\R^{dN}$-valued process $(\tilde \xi_t)_{t\in [0,T]}$ and the family of nonnegative functions $(\tilde{u}_t)_{t\in [0,T]}$ defined on $\R^d$ constructively such that 
\begin{equation}
\label{eq:tildeYu}
\left \{\begin{array}{l}
%\tilde{u}_0=u_0\\
\tilde \xi^{i,N}_{t}=\tilde \xi^{i,N}_{0}+\int_0^t \Phi(r(s),\tilde{\xi}^{i,N}_{r(s)},\tilde{u}_{r(s)}(\tilde \xi^{i,N}_{r(s)}))dW^i_{s}+\int_0^t g(r(s),\tilde{\xi}^{i,N}_{r(s)},\tilde{u}_{r(s)}(\tilde \xi^{i,N}_{r(s)}))ds \\ %\quad \textrm{with}\ (\tilde \xi^{i,N}_0)_{i=0,\cdots N}\ \textrm{i.i.d.}\ \sim \tilde{u}_0\\
\tilde{\xi}^{i,N}_0 = Y^i_0 \\
\tilde{u}_{t}(y)=\frac{1}{N}\sum_{j=1}^N K(y-\tilde \xi^{j,N}_{t})\exp\big \{\int_0^t \Lambda(r(s),\tilde \xi^{j,N}_{r(s)},\tilde{u}_{r(s)}(\tilde \xi^{j,N}_{r(s)}))\,ds\big \} \ ,\ \textrm{for any}\ t\in ]0,T],\\
\tilde{u}_0 = K * \zeta_0,
\end{array}
\right .
\end{equation}
where $r:\,s\in [0,T]\,\mapsto r(s)\in \{t_0,\cdots t_n\}$ is the piecewise constant function such that $r(s)=t_k$ when $s\in [t_k,t_{k+1}[$. We can observe that $(\tilde{\xi}^{i,N})_{i=1,\cdots,N}$ is an adapted and continuous process. 
The interacting particle system $(\tilde \xi^{i,N})_{i=1,\cdots N}$ can be simulated perfectly at the discrete 
instants $(t_k)_{k=0,\cdots,n}$ via  independent standard and centered Gaussian random variables.
 We will show that this interacting particle system provides an approximation to the solution  $(\xi^{i,N})_{i=1,\cdots N}$, 
 of  system~\eqref{eq:XIi}, which converges at a rate bounded by
 $  \sqrt{\delta t}$, up to a multiplicative constant.
\begin{prop}
\label{prop:DiscretTime}
Let us suppose the validity of Assumption \ref{ass:main2}.
The time discretized particle system~\eqref{eq:tildeYu} converges to the original particle system~\eqref{eq:XIi}. More precisely, for all $t \in [0,T]$, the following estimates hold:
\begin{equation}
\label{eq:tildeYMajor}
\E[\Vert \tilde{u}_t-u_t^{S^N(\xi)}\Vert_{\infty}^2]+\sup_{i=1,\cdots, N} \E\left [\sup_{s\leq t}\vert \tilde  \xi^{i,N}_{s}-  \xi^{i,N}_{s}\vert^2\right ]\leq C\delta t\ ,
\end{equation}
where $C$ is a finite positive constant only depending on $M_K,M_{\Lambda},L_K,L_{\Lambda}, m_\Phi, m_g, T$.  \\
If we assume moreover that $K\in W^{1,2}(\R^d)$, then 
%the following Mean Integrated Squared Error (MISE) estimate holds: 
\begin{equation}
\label{eq:tildeYMajorBis}
\E[\Vert \tilde{u}_t-u_t^{S^N(\xi)}\Vert_{2}^2]\leq C\delta t \ , \quad t \in [0,T],
\end{equation}
where $C$ is a finite positive constant only depending on $M_K,M_{\Lambda},m_{\Phi}, m_g,L_K,L_{\Lambda},T$ and $\Vert \nabla K\Vert_2$. 
 
The left-hand side of \eqref{eq:tildeYMajorBis} is generally known, as Mean Integrated Squared Error (MISE).
\end{prop}

The result below states the convergence of $\tilde{u}_t$ to $u^{m^0}_t$ when $\delta t \rightarrow 0$ and $N \rightarrow + \infty$, with an explicit rate of convergence.
\begin{thm}
\label{CS8}
We suppose Assumption \ref{ass:main2}. 
 We indicate by $m^0$  the law of $Y$, where $(Y,u)$ is the solution  of \eqref{eq:NSDE}.
The time discretized particle system~\eqref{eq:tildeYu} converges to the solution of \eqref{eq:NSDE}.
 More precisely, we have the following. We suppose 
${\mathcal F}(K) \in L^1$ (resp. $K\in W^{1,2}(\R^d)$).
 There exists a real constant $C > 0$ such that for all $t \in [0,T]$,
\begin{equation}
\label{eq:CS8}
 \E[ \Vert u^{m^0}_t - \tilde{u}_t \Vert_{\infty} ^2] \leq C ( \delta t + \frac{1}{N} ),
\end{equation}
(respectively 
\begin{equation}
\label{eq:CS8bis}
 \E[ \Vert u^{m^0}_t - \tilde{u}_t \Vert_{2} ^2] \leq C ( \delta t + \frac{1}{N} ) \ ).
\end{equation}

\end{thm}
\begin{rem} \label{eq:R68} 
When $\Lambda=0$
and  $\Phi$ and $g$ are infinitely differentiable with all 
     derivatives being bounded,  Corollary 1.1 of   \cite{KO97}
states that,  for fixed smooth test function
with polynomial growth $\varphi$, one has
\begin{equation} \label{EPhi}
  \mathbb{E}(\vert \langle S^N(\tilde \xi_t)- m^0_t, \varphi \rangle\vert)
\le C_\varphi (\frac{1}{\sqrt{N}} + \delta t), \quad {\rm where \ again} \quad
S^N(\mathbf{\tilde{\xi_t}}):=\frac{1}{N} \sum_{i=1}^N \delta_{\tilde{\xi}_t^{i,N}}\ .
\end{equation}
This leads reasonnably to the conjecture that the rate in \eqref{eq:CS8} is not optimal and
it could be replaced by  $(\delta t)^2 + \frac{1}{N}$. This intuition will be confirmed
by numerical simulations in Section \ref{SNum}.
\end{rem}

\begin{proof}
We first observe that for all $t \in [0,T]$, for $p = \infty, 2$
\begin{eqnarray}
\label{eq:CS8_2}
\E[ \Vert u^{m^0}_t - \tilde{u}_t \Vert_{p} ^2] & \leq & 2 \E[ \Vert u^{m^0}_t - u_t^{S^N(\xi)}\Vert_{p} ^2] +
2 \E[\Vert u_t^{S^N(\xi)} -  \tilde{u}_t\Vert_{p}^2] \ .
\end{eqnarray}
%The first term in the r.h.s. of \eqref{eq:CS8_2} comes from the (strong) convergence of the particle system $(\xi^{i,N})_{i=1,\cdots,N}$ to $(Y^i)_{i=1,\cdots,N}$ 
%whose  convergence is of order $\frac{1}{N}$, see Theorem \ref{TPC}, inequality \eqref{eq:xiYuFinalb1}. \\
The first term in the r.h.s. of \eqref{eq:CS8_2} is bounded by 
$\frac{C}{N}$ using  Theorem \ref{TPC}, inequality \eqref{eq:xiYuFinalb1}
(respectively \eqref{eq:xiYuFinal3}). \\
The  second term of the same inequality is controlled  by $C \delta t$,
%time discretization whose expected  squared error is of order $ \delta t$ Proposition \ref{prop:DiscretTime}, inequality \eqref{eq:tildeYMajor}.
through  Proposition \ref{prop:DiscretTime}, inequality \eqref{eq:tildeYMajor} (resp.
\eqref{eq:tildeYMajorBis}).
\end{proof}
%\begin{rem} \label{RPIDE}
%We keep in mind the probability measure $m_0$ defined at Section \ref{SChaos},
%which is the law of processes $Y^i$, solutions of \eqref{eq:Yi}.
% We claim that $\tilde{u}$ can be used as a numerical approximation to the function $u^{m_0}$;
%we remind that by Theorem 6.1 of \cite{LOR1}, $u^{m_0}$ is associated with a solution $\gamma^{m_0}$
%of the PIDE~\eqref{epide} via the relation $u^m = K * \gamma^m$.
%
%The committed expected squared error 
%$\E[ \Vert u^{m_0}_t - \tilde{u}_t \Vert_{\infty} ^2]$ 
%is lower than $C(T)(\delta t + 1/N)$, for a given finite constant $C(T)$. Indeed, it is bounded as follows:
%%the error $\E[ \Vert u^{m_0}_t - \tilde{u}_t \Vert_{\infty} ^2]$ as follows
%$$
% \E[ \Vert u^{m_0}_t - \tilde{u}_t \Vert_{\infty} ^2] \leq 2 \E[ \Vert u^{m_0}_t - u_t^{S^N(\xi)}\Vert_{\infty} ^2] +
% 2 \E[\Vert u_t^{S^N(\xi)} -  \tilde{u}_t\Vert_{\infty}^2].
%$$
%\end{rem}

The proof of Proposition  \ref{prop:DiscretTime} relies on similar techniques used to prove Theorem \ref{TPC}.
%is close to the one of  Theorem \ref{TPC}.
The idea is first to estimate through Lemma \ref{lem:Discrete}
 the perturbation error due to the time discretization scheme of the SDE and of the integral
appearing in the exponential weight
in    system~\eqref{eq:tildeYu}.
%  and    the integral
% appearing in the linking equation of~\eqref{eq:tildeYu}. 
Later the propagation of this error  through the dynamical system~\eqref{eq:XIi}
will be controlled via  Gronwall's lemma. 
Lemma \ref{lem:Discrete} below will be proved in the Appendix.

\begin{lem}
\label{lem:Discrete}
%Under the same assumptions of Proposition~\ref{prop:DiscretTime},
Let us suppose the validity of Assumption \ref{ass:main2}.
There exists a finite constant $C>0$ only depending on $T,M_K,m_{\Phi},m_g,L_K,L_\Phi,L_g$ and $M_{\Lambda},L_{\Lambda}$ such that for any $t\in [0,T]$, 
%ANTHONY Y A-T-IL ENCORE  $\sup_{s \in [0,T]} (\vert \Phi(s,0,0) \vert + \vert g(s,0,0)$ OU AU MOINS AUTRE CHOSE? \\
\begin{eqnarray}
\label{eq:vtilde}
\E[\vert \tilde \xi ^{i,N}_{r(t)}-\tilde \xi^{i,N}_{t}\vert ^2]\leq C\delta t
 \\
\label{E82}
\E[\Vert \tilde{u}_{r(t)}-\tilde{u}_t\Vert^2_{\infty}]\leq C\delta t\\
\label{E83}
\E[\Vert \tilde{u}_{r(t)}-u^{S^N(\tilde \xi)}_t\Vert^2_{\infty}]\leq C\delta t\ .
\end{eqnarray}
\end{lem} 

\begin{proof}[Proof of Proposition~\ref{prop:DiscretTime}.]
All along this proof,  $C$ will denote a positive constant that only  depends on \\
 $T ,M_K,m_{\Phi}, m_g, L_K,L_\Phi, L_g$ and $M_{\Lambda}$,$L_{\Lambda}$ 
and that can change from line to line. Let us fix $t \in [0,T]$.
\begin{itemize}
\item We begin by considering inequality~\eqref{eq:tildeYMajor}.
We first fix $1 \le i \le N$. 
By    \eqref{E82} and \eqref{E83} in Lemma~\ref{lem:Discrete} and
% \eqref{eq:uu'2} 
\eqref{eq:uu'}
in Proposition \ref{lem:uu'}, we obtain
\begin{eqnarray}
\label{eq:tildevMajor}
\E[\Vert \tilde{u}_t-u^{S^N(\xi)}_t\Vert_\infty^2]
&\leq &
\E \Big[ \Big( \Vert \tilde{u}_t-\tilde{u}_{r(t)} \Vert_{\infty} + \Vert \tilde{u}_{r(t)}- u^{S^N(\tilde{\xi})}_t \Vert_{\infty} + \Vert u^{S^N(\tilde{\xi})}_t - u^{S^N(\xi)}_t \Vert_{\infty} \Big)^2 \Big] \nonumber \\
%2\E[\Vert \tilde{u}_t-u_t^{S^N( \tilde \xi)}\Vert_{\infty}^2] + 2\E[\Vert u_t^{S^N(\tilde \xi)}-u_t^{S^N(\xi)}\Vert_\infty^2]\nonumber \\
&\leq & 3(\E[\Vert \tilde{u}_t-\tilde{u}_{r(t)} \Vert_{\infty}^2] + \E[\Vert \tilde{u}_{r(t)}-u_t^{S^N( \tilde \xi)}\Vert_{\infty}^2] + \E[\Vert u_t^{S^N(\tilde \xi)}-u_t^{S^N(\xi)}\Vert_\infty^2])\nonumber \\
%4(\E[\Vert \tilde{u}_t-\tilde{u}_{r(t)} \Vert_{\infty}^2] + \E[\Vert \tilde{u}_{r(t)}-u_t^{S^N( \tilde \xi)}\Vert_{\infty}^2] ) + 2\E[\Vert u_t^{S^N(\tilde \xi)}-u_t^{S^N(\xi)}\Vert_\infty^2]\nonumber \\
& \leq & C\delta t +C\E [\vert W_t\big (S^N(\tilde \xi),S^N(\xi)\big )\vert ^2]\nonumber \\
&\leq & 
C\delta t +C\sup_{i=1,\cdots, N} \E[\sup_{s\leq t} \vert \tilde \xi^{i,N}_s-\xi^{i,N}_s\vert ^2]\ ,
\end{eqnarray}
where  the function $u^{S^N(\tilde{\xi})}$ makes sense  since $ \tilde{\xi}$ has almost surely continuous trajectories
 and so $S^N(\tilde{\xi})$ is a random probability  
%by Remark \ref{RDelta} with values 
in $\shp(\shc^d)$. \\
Besides, by the second assertion of Lemma~~\ref{lem:yy'},   
setting $Y' := \tilde{\xi}^{i,N}$, $r_1(s) = r(s)$ and $Y := \xi^{i,N}$, $r_2(s) = s$,
we get
\begin{eqnarray}
\label{E310}
\E[\sup_{s\leq t} \vert \tilde \xi^{i,N}_s-\xi^{i,N}_s\vert ^2]
&\leq & C \E\left [ \int_0^t \vert \tilde{u}_{r(s)}(\tilde \xi^{i,N}_{r(s)})-u_s^{S^N(\xi)}(\xi^{i,N}_s)\vert^2\,ds\right ] + C \int_0^t \E \left[ \vert \tilde{\xi}^{i,N}_{r(s)} - \tilde{\xi}^{i,N}_{s} \vert^2 \right] ds + C \delta t \ . \nonumber \\ 
\end{eqnarray}
Concerning the first term in the r.h.s. of \eqref{E310}, we have for all $s \in [0,T]$
\begin{eqnarray}
\label{E311}
\vert \tilde{u}_{r(s)}(\tilde \xi^{i,N}_{r(s)})-u_s^{S^N(\xi)}(\xi^{i,N}_s)\vert^2 & \leq & 2 \vert \tilde{u}_{r(s)}(\tilde \xi^{i,N}_{r(s)})-u_s^{S^N(\xi)}(\tilde{\xi}^{i,N}_{r(s)})\vert^2 + 2 \vert u_s^{S^N(\xi)}(\tilde \xi^{i,N}_{r(s)})-u_s^{S^N(\xi)}(\xi^{i,N}_s)\vert^2 \nonumber \\
& \leq & 2 \Vert \tilde{u}_{r(s)} - u^{S^N(\xi)}_s \Vert_{\infty}^2 + 2 C \vert \tilde{\xi}^{i,N}_{r(s)} - \xi^{i,N}_s \vert^2 \ ,
\end{eqnarray}
where the second inequality above follows by \eqref{eq:uu'} 
in  Proposition \ref{lem:uu'},  setting $m = m' = S^N(\xi)$.
 %(Lipschitz property of the function $u^{S^N(\xi)}$).
 Consequently, by \eqref{E310}
\begin{eqnarray}
\E[\sup_{s\leq t} \vert \tilde \xi^{i,N}_s-\xi^{i,N}_s\vert ^2]
&\leq & 
C \left \{ \E\left [ \int_0^t \Vert \tilde{u}_{r(s)}-u_s^{S^N(\xi)}\Vert^2_{\infty} \,ds \right ]  + \int_0^t \E\left [ \vert \tilde{\xi}^{i,N}_{r(s)}-\xi_{s}^{i,N} \vert^2 \right] \,ds  + \delta t \right \} \nonumber \\  
& \leq & C \left \{ \E\left [ \int_0^t \Vert \tilde{u}_{r(s)}- \tilde{u}_s \Vert^2_{\infty}\,ds\right ] + \E\left [ \int_0^t \Vert \tilde{u}_s-u_s^{S^N(\xi)}\Vert^2_{\infty}\,ds\right ]  \right . \nonumber \\ 
& & \left . + \E\left [\int_0^t\vert \tilde \xi^{i,N}_{r(s)}- \tilde \xi^{i,N}_s\vert ^2\,ds\right ] + \E\left [\int_0^t\vert \tilde \xi^{i,N}_{s}- \xi^{i,N}_s\vert ^2\,ds\right ] + \delta t \right \} \ .
\end{eqnarray}
Using  inequalities \eqref{eq:vtilde} and \eqref{E82} in Lemma~\ref{lem:Discrete}, for all $t \in [0,T]$,
we obtain
\begin{equation}
\label{eq:xitildeMajor}
\sup_{i=1,\cdots, N}\E[\sup_{s\leq t} \vert \tilde \xi^{i,N}_s-\xi^{i,N}_s\vert ^2]
\leq C\delta t +C \int_0^t \left [\E[\Vert \tilde{u}_{s}-u_s^{S^N(\xi)}\Vert^2_{\infty}]+ \sup_{i=1,\cdots, N}\E[\sup_{\theta \leq s} \vert \tilde \xi^{i,N}_\theta-\xi^{i,N}_\theta\vert ^2]\right ]\,ds.
\end{equation}
Gathering the latter inequality together with~\eqref{eq:tildevMajor} yields 
\begin{eqnarray}
\E[\Vert \tilde{u}_t-u^{S^N(\xi)}_t\Vert_\infty^2] + \sup_{i=1,\cdots, N}\E[\sup_{s\leq t} \vert \tilde \xi^{i,N}_s-\xi^{i,N}_s\vert ^2] & \leq & C\delta t + 2C\sup_{i=1,\cdots, N} \E[\sup_{s\leq t} \vert \tilde \xi^{i,N}_s-\xi^{i,N}_s\vert ^2] \nonumber \\
& \leq & C\delta t \nonumber \\ 
& & + \; C \int_0^t \Big [\E[\Vert \tilde{u}_{s}-u_s^{S^N(\xi)}\Vert^2_{\infty}] \nonumber \\
&& + \sup_{i=1,\cdots, N}\E[\sup_{\theta \leq s} \vert \tilde \xi^{i,N}_\theta-\xi^{i,N}_\theta\vert ^2]\Big ] ds \ .
\end{eqnarray}
Applying Gronwall's lemma to the function 
$$
t\mapsto \sup_{i=1,\cdots, N}\E[\sup_{s\leq t} \vert \tilde \xi^{i,N}_s-\xi^{i,N}_s\vert ^2]+\E[\Vert \tilde{u}_{t}-u_t^{S^N(\xi)}\Vert^2_{\infty}]
$$ 
ends the proof of  \eqref{eq:tildeYMajor}. 

\item We  focus now on~\eqref{eq:tildeYMajorBis}.  First we observe that 
\begin{equation}
\label{eq:vTildeL2}
\E [\Vert \tilde{u}_t-u^{S^N(\xi)}_t\Vert_2^2]\leq 2\E [\Vert \tilde{u}_t-u_t^{S^N(\tilde \xi)}\Vert_2^2]+2\E [\Vert u_t^{S^N(\tilde \xi)}-u_t^{S^N(\xi)}\Vert_2^2]\ .
\end{equation}
Using successively item 4. of Proposition~\ref{lem:uu'}, Remark \ref{RADelta} and 
inequality~\eqref{eq:tildeYMajor},
we can bound the second term on the r.h.s. of \eqref{eq:vTildeL2}  as follows:
\begin{eqnarray}
\label{E815}
\E [\Vert u_t^{S^N(\tilde \xi)}-u_t^{S^N(\xi)}\Vert_2^2]
&\leq &
C\E [\vert W_t\big (S^N(\tilde \xi),S^N(\xi)\big )\vert^2] \nonumber \\
&\leq &
C\sup_{i=1,\cdots, N}\E[\sup_{s\leq t}\vert \tilde \xi^{i,N}_s-\xi^{i,N}_s\vert^2] \nonumber \\
&\leq &
C\delta t\ .
\end{eqnarray} 
 
To simplify the notations, we introduce the real valued random variables 
\begin{equation}
\label{eq:VVtildeDef}
V_t^i := e^{\int_0^t\Lambda\big (s,\tilde \xi^{i,N}_{s},u^{S^N(\tilde{\xi})}_{s}(\tilde \xi^{i,N}_{s})\big )ds}\quad \textrm{and}\quad \tilde V_t^i := e^{\int_0^t\Lambda\big (r(s),\tilde \xi^{i,N}_{r(s)},\tilde{u}_{r(s)}(\tilde \xi^{i,N}_{r(s)})\big )ds}\ ,
\end{equation}
defined for any $i=1,\cdots N$ and $t\in [0,T]$. \\
Concerning the first term on the r.h.s. of~\eqref{eq:vTildeL2}, inequality \eqref{E942} of Lemma \ref{lem:ViVi'} 
in the Appendix gives for all $y \in \R^d$
\begin{eqnarray}
\label{E816}
\vert \tilde{u}_{t}(y)-u_t^{S^N(\tilde \xi)}(y)\vert^2 \leq  \frac{M_K}{N}\sum_{i=1}^N K(y-\tilde \xi^{i,N}_t)  \vert \tilde V_t^i-V_t^i \vert^2 \ .
\end{eqnarray}
Integrating the inequality \eqref{E816} with respect to $y$,  yields
$$
\Vert \tilde{u}_t - u_t^{S^N(\tilde \xi)} \Vert_2^2 = \int_{\R^d} \vert \tilde{u}_{t}(y)-u_t^{S^N(\tilde \xi)}(y)\vert^2\,dy
\leq 
\frac{M_K}{N}\sum_{i=1}^N   \vert \tilde V_t^i
-
V_t^i \vert^2\ ,
$$
which, in turn, implies
\begin{eqnarray}
\label{E817}
\E \left[ \Vert \tilde{u}_t - u^{S^N(\tilde{\xi})}_t \Vert_2^2 \right] \leq \frac{M_K}{N}\sum_{i=1}^N   \E \left[ \vert \tilde V_t^i - V_t^i \vert^2 \right] \ .
\end{eqnarray}
Using successively item $1.$ of Lemma \ref{lem:ViVi'} and inequality \eqref{eq:vtilde} of~Lemma~\ref{lem:Discrete},
 for all $i \in \{1,\cdots,N\}$, we obtain
\begin{eqnarray}
\label{E818}
 \E[\vert \tilde V_t^i-V_t^i \vert^2]
 & \leq &  C\delta t + C\E \left[ \int_0^t \vert \tilde{\xi}^{i,N}_{r(s)} - \tilde{\xi}^{i,N}_{s}  \vert^2   \ ds \right] + C \E \left[ \int_0^t \vert \tilde{u}_{r(s)}(\tilde \xi^{i,N}_{r(s)})-u^{S^N(\tilde \xi)}_s(\tilde \xi^{i,N}_s)\vert^2 ds \right] \nonumber \\
&\leq & C\delta t + C \E \left[ \int_0^t \vert \tilde{u}_{r(s)}(\tilde \xi^{i,N}_{r(s)})-u^{S^N(\tilde \xi)}_s(\tilde \xi^{i,N}_s)\vert^2 ds \right] \nonumber\\
& \leq & C\delta t + C \E \left[ \int_0^t \vert \tilde{u}_{r(s)}(\tilde \xi^{i,N}_{r(s)})-u^{S^N(\tilde \xi)}_s(\tilde \xi^{i,N}_{r(s)})\vert^2 ds \right] \nonumber \\
&& + \; C \E \left[ \int_0^t \vert u_{s}^{S^N(\tilde{\xi})}(\tilde \xi^{i,N}_{r(s)})-u^{S^N(\tilde \xi)}_s(\tilde \xi^{i,N}_{s})\vert^2 ds \right] \nonumber \\
&\leq & C\delta t +C\int_0^t \left [ \E [\Vert \tilde{u}_{r(s)}-u^{S^N(\tilde \xi)}_s\Vert_\infty^2] + \E[\vert \tilde \xi^{i,N}_{r(s)}-\tilde \xi^{i,N}_s\vert^2] \right ]\, ds \nonumber\\
&\leq &
C\delta t +C\int_0^t \E [\Vert \tilde{u}_{r(s)}-u^{S^N(\tilde \xi)}_s\Vert_\infty^2] \, ds\ ,
\end{eqnarray}
where the fourth inequality above follows from  Proposition \ref{lem:uu'}, see \eqref{eq:uu'}. Consequently using~\eqref{E818} and inequality \eqref{E83} of Lemma~\ref{lem:Discrete},  \eqref{E817} becomes
\begin{eqnarray}
\label{EFinal_Discret}
\E [\Vert \tilde{u}_t-u_t^{S^N(\tilde \xi)}\Vert_2^2]
\leq \frac{C}{N}\sum_{i=1}^N \E [\vert \tilde V_t^i - V_t^i \vert^2] \underbrace{\leq}_{\eqref{E818}} C\delta t +C\int_0^t \E [\Vert \tilde{u}_{r(s)}-u^{S^N(\tilde \xi)}_s\Vert_\infty^2] \underbrace{\leq}_{\eqref{E83}}  C\delta t.
\end{eqnarray}
Finally, injecting \eqref{EFinal_Discret} and \eqref{E815} in \eqref{eq:vTildeL2} yields
$$
\E [\Vert \tilde{u}_t-u^{S^N(\xi)}_t\Vert_2^2] \leq C \delta t \ ,
$$
\end{itemize}
which ends the proof of Proposition \ref{prop:DiscretTime}.
\end{proof}

\subsection{Algorithm description}
In this section, we describe precisely the algorithm relying on the time-discretization \eqref{eq:tildeYu} of the interacting particle system \eqref{eq:XIi}. Let $v_0$ be the law density of $Y_0$ where $Y$ is the solution of \eqref{eq:NSDE}. In the sequel, we will make use of the same notations as in previous section. In particular, $0=t_0\leq \cdots\leq t_k=k\delta t\leq \cdots \leq t_n=T$ is a regular time grid with $\delta t=T/n$. %For fixed $h > 0$,
 We consider a real-valued function $K : \R^d \rightarrow \R$ being a   mollifier  depending on some
 bandwith parameter $\varepsilon$. \\
%where for all $x \in \R^d$, $K^h(x) = \frac{1}{h^d} \phi^d(\frac{x}{h})$.
%We also recall that the time-discretized interacting particle system \eqref{eq:tildeYu} is denoted by $(\tilde{\xi}^{i,N})_{i=1,\cdots,N}$ and the continuous system \eqref{eq:XIi} is denoted by $(\xi^{i,N})_{i=1,\cdots,N}$.
\begin{description}
	\item[Initialization] 
	for $k=0$.
	\begin{enumerate} 
		\item Generate $(\tilde{\xi}_{t_0}^{i,N})_{i=1,..,N}$  i.i.d.$\sim$ $v_0(x)dx$;
		\item set $G_0^i := 1$, $i = 1,\cdots,N$;
		\item set $\tilde{u}_{t_0}(\cdot) := (K \ast v_0)(\cdot)$;   
	\end{enumerate}
	\item[Iterations]  
	for k = 0, ..., n-1.
		\begin{itemize}
			\item Independently for each particle  $ { \tilde{\xi}^{j,N}_{t_k}}$ for $j=1,\cdots N$, 
			$$\tilde \xi^{j,N}_{t_{k+1}}= {\tilde{\xi}^{j,N}_{t_k}} + \Phi(t_{k},{\tilde{\xi}^{j,N}_{t_k}},{\tilde{u}}_{t_{k}}({\tilde{\xi}^{j,N}_{t_k}})) \sqrt{\delta t} \epsilon^j_{k+1} + g(t_{k},{\tilde{\xi}^{j,N}_{t_k}},{\tilde{u}}_{t_{k}}({\tilde{\xi}^{j,N}_{t_k}})) \delta t \; ,$$
			where $(\epsilon^j_{k})_{j=1, \cdots, N, k=1,\cdots n}$ is a sequence of i.i.d centered and standard Gaussian variables; 
		\end{itemize}
 
		\begin{itemize}
			\item set for $j=1,\cdots N$, $$G_{k+1}^j := G_{k}^j \times \exp \left( \Lambda(t_{k},{ \tilde{\xi}^{j,N}_{k}},{\tilde{u}}_{t_k}({ \tilde{\xi}^{j,N}_{t_k}})) \delta t  \right);$$
			\item set $${\tilde{u}}_{t_{k+1}}(\cdot) = {\displaystyle \frac{1}{N} \sum_{j=1}^{N}  G_{k+1}^j \times K(\cdot - { \tilde{\xi}^{j,N}_{t_{k+1}}})}. $$        
  		\end{itemize}
\end{description}
\begin{rem}
For a fixed $k \in \{0,\cdots,n-1 \}$, we observe that the simulation of the $j$-th particle $\tilde{\xi}^{j,N}_{t_{k+1}}$ at time $t_{k+1}$ involves the whole particle system  through the evaluation of $\tilde{u}_{t_k}(\tilde \xi_{t_k}^{j,N})$,
which implies a complexity of the algorithm of order $n N^2$.
\end{rem}

\section{Numerical results}
\label{SNum}

\subsection{Preliminary considerations}

One motivating issue of this section is how the interacting particle system $\xi := (\xi^{i, N,\varepsilon})$ defined in \eqref{eq:XIi} with
 $K = K^{\varepsilon}$, $K^{\varepsilon}(x) := \frac{1}{\varepsilon^d}\phi^d(\frac{x}{\varepsilon})$ for some mollifier $\phi^d$, can be used to approach 
the solution $v$ of the PDE 
\begin{equation}
\label{pdeLim}
\left \{
\begin{array}{l}
\partial_t v = \frac{1}{2} \displaystyle{\sum_{i,j=1}^d} \partial_{ij}^2 \left( (\Phi \Phi^t)_{i,j}(t,x, v) v \right) - div \left( g(t,x, v) v \right) +\Lambda(t,x,v)  v \\
 v(0,x) = v_0 \ , %\mathcal{L}(Y_0) \ .
\end{array}
\right .
\end{equation}
to which we can reasonably expect that \eqref{epide} converges when $K^{\varepsilon} \xrightarrow[\varepsilon \rightarrow 0]{} \delta$. \\
Two significant parameters, i.e.  $\varepsilon \rightarrow 0$, $N \rightarrow + \infty$, intervene.
We expect to approximate $v$ by $u^{\varepsilon,N}$, which is the solution of the linking equation \eqref{NSDE3}, associated with the empirical measure $m = S^N(\xi)$. To this purpose, we want to control empirically the 
 Mean Integrated Squared Error (MISE) between the solution $v$ of \eqref{pdeLim} and the particle approximation $u^{\varepsilon,N}$, i.e. for $t \in [0,T]$,
\begin{eqnarray} 
\label{eq:MISE}
\E[ \Vert u_t^{\varepsilon,N} -  v_t \Vert_2^2] \leq 2 \E[ \Vert u_t^{\varepsilon,N} -  u_t^{\varepsilon} \Vert_2^2] + 
2  \Vert u_t^{\varepsilon} -  v_t \Vert_2^2,
\end{eqnarray}
where $u^{\varepsilon} = u^{m^0}$ with $K = K^{\varepsilon}$, $m^0$ being the common law of processes $Y^i, \; 1 \leq i \leq N$ in \eqref{eq:Yi}. Even though the second term in the r.h.s. of \eqref{eq:MISE} does not explicitely involve the number of particles $N$, the first term crucially depends on both parameters $\varepsilon, N$. The behavior of the first term relies on the propagation of chaos. This phenomenon has been observed in Corollary \ref{CCP}, which is a consequence of Theorem \ref{TPC}, for a fixed $\varepsilon > 0$, when $N \rightarrow + \infty$.
 According to Theorem \ref{TPC}, the first error term on the r.h.s. of the above inequality can be bounded by $\frac{C(\varepsilon)}{N}$. \\
Concerning the second error term, no result is available but we expect that it converges to zero when $\varepsilon \rightarrow 0$. To control the MISE, it remains to determine a relation $N \mapsto \varepsilon(N)$ such that
\begin{equation}
\label{eq:tradeoff}
\varepsilon(N) \xrightarrow[N \rightarrow + \infty]{} 0 \quad \textrm{ and } \quad \frac{C(\varepsilon(N))}{N} \xrightarrow[N \rightarrow + \infty]{} 0 \ .
\end{equation}
When the coefficients $\Phi$, $g$ and the initial condition are smooth with $\Phi$ non-degenerate and $\Lambda \equiv 0$ (i.e. in the conservative case), Theorem 2.7 of \cite{JourMeleard} gives a description of such a relation. \\
In our empirical analysis, we have concentrated on a test case, for which we have an explicit solution.
\\
 We first illustrate the chaos propagation for fixed $\varepsilon > 0$, i.e. the result of Theorem \ref{TPC}. On the other hand, we give an empirical insight concerning the following:
\begin{itemize}
\item the asymptotic behavior of the second error term in inequality \eqref{eq:MISE} for $\varepsilon \rightarrow 0$;
\item the tradeoff $N \mapsto \varepsilon(N)$ verifying \eqref{eq:tradeoff}.
\end{itemize}
Moreover, the simulations reveal two behaviors regarding the chaos propagation intensity.

\subsection{The target PDE}
%%%%%%%%%%%%%%%%%%%%%%%%%%%%%%%%%%%

We describe now the test case.
For a given triple $(m,\mu , A)\in ]1,\infty[\times \R^d\times \R^{d \times d}$
%\mathcal{M}_d(\R)$, 
we consider the following nonlinear PDE of the form~\eqref{pdeLim}:
\begin{equation}
\label{eq:pdev}
\left \{
\begin{array}{lll}
\partial_t v&=&{\displaystyle \frac{1}{2}\sum_{i,j=1}^d \partial_{i,j}^2 \big (v(\Phi \Phi^t)_{i,j}(t,x,v)\big )}-div\big (vg(t,x,v)\big )+v\Lambda(t,x,v) \ ,\\
v(0,x)&=&B_m(2,x)f_{\mu,A}(x)\quad\textrm{for all}\ x\in \R^d \ ,
\end{array}
\right .
\end{equation}
where the functions $\Phi\,,\,g\,,\, \Lambda$ defined on $[0,T]\times \R^d\times \R$ are such that 
\begin{equation}
\label{eq:Phi}
\Phi(t,x,z)=f_{\mu,A}^{\frac{1-m}{2}}(x)z^{\frac{m-1}{2}}I_d\ ,
\end{equation} 
$I_d$ denoting the identity matrix in $\R ^{d \times d}$, 
\begin{equation}
\label{eq:gL}
g(t,x,z)=f_{\mu,A}^{1-m}(x)z^{m-1}\frac{A+A^t}{2}(x-\mu )\ ,
\quad
\textrm{and} 
\quad
\Lambda(t,x,z)=f_{\mu,A}^{1-m}(x)z^{m-1}Tr\left(\frac{A+A^t}{2}\right).
\end{equation}
%with $A_{sym}$ denoting the symmetric matrix $A_{sym}=\frac{A+A^t}{2}$ 
Here $f_{\mu,A}:\R^d \rightarrow \R $ is given by 
\begin{equation}
\label{eq:f}
f_{\mu ,A}(x)= C e^{-\frac{1}{2}\langle x-\mu,A(x-\mu)\rangle}\ ,
\quad\textrm{normalized by}\quad  
{\displaystyle C=\left [ \int_{\R^d} B_m(2,x)e^{-\frac{1}{2} (x-\mu) \cdot A(x-\mu)}\right ]^{-1}}
\end{equation}
 and    $B_m$ is the $d$-dimensional Barenblatt-Pattle density associated to $m>1$, i.e. 
\begin{equation}
\label{eq:Barenblatt} 
B_m(t,x)= \frac{1}{2}(D -\kappa  t^{-2\beta }\vert x\vert)_+^{\frac{1}{m-1}}t^{-\alpha},
\end{equation}
with
$\alpha =\frac{d}{(m-1)d+2}\ ,$   $\beta =\frac{\alpha }{d} \ ,$  $\kappa =\frac{m-1}{m}\beta$  and 
 $D =[2\kappa ^{-\frac{d}{2}}\frac{\pi^{\frac{d}{2}}\Gamma (\frac{m}{m-1})}{\Gamma (\frac{d}{2}+\frac{m}{m-1})} ]^{\frac{2(1-m)}{2+d(m-1)}}\ .
$

In the specific case where $A$ is the zero matrix of $\R^{d \times d}$, then $f_{\mu,A}\equiv 1$; $g\equiv 0$ and $\Lambda \equiv 0$. Hence, we recover the conservative porous media equation, whose explicit solution is
$$
v(t,x)=B_m(t+2,x)\ ,\quad \textrm{for all} \ (t,x)\in [0,T]\times \R^d,
$$
see~\cite{Barenb}.
For general values of $A\in \R^{d \times d}$, extended calculations produce
 the following explicit solution
\begin{equation}
\label{eq:sol}
v(0,\cdot) = v_0(\cdot) \quad \textrm{and} \quad v(t,x)=B_m(t+2,x)f_{\mu,A}(x)\ ,\quad \textrm{for all} \ (t,x)\in [0,T]\times \R^d \; ,
\end{equation}
 of~\eqref{eq:pdev}, which is non conservative. 

\subsection{Details of  the implementation}
%%%%%%%%%%%%%%%%%%%%%%%%%%%%%%%%%%%%%%%%%

Once fixed the number $N$ of particles, we have run $M=100$ i.i.d. particle systems producing $(u^{\varepsilon,N,i})_{i=1,\cdots M}$, which are $M$ i.i.d. realizations
 of $u^{\varepsilon,N}$ introduced just after \eqref{pdeLim}. %estimates $(u^{\varepsilon,N,i})_{i=1,\cdots M}$.
The MISE is then approximated by the  Monte Carlo approximation 
\begin{equation}
\label{eq:MISEApprox}
\E[ \Vert u_t^{\varepsilon,N}-v_t\Vert_2^2 ] \approx \frac{1}{MQ}\sum_{i=1}^M \sum_{j=1}^Q \vert u_t^{\varepsilon,N,i}(X^j)-v_t(X^j)\vert ^2 v^{-1}(0,X^j)\ ,\quad\textrm{for all}\ t\in [0,T]\ ,
\end{equation}
where $(X^j)_{j=1,\cdots, Q=1000}$ are i.i.d $\R^d$-valued random variables with common density $v(0,\cdot)$. 
In our simulation, we have chosen $T=1$, $m=3/2$, $\mu=0$ and $A=\frac{2}{3} I_d$.
 $K^{\varepsilon}=\frac{1}{\varepsilon^d}\phi^d(\frac{\cdot}{\varepsilon})$ with $\phi^d$ being the standard and centered Gaussian density. 
% The initial condition $v(0,\cdot)$ is perfectly simulated using a rejection algorithm with a Gaussian instrumental distribution. 

% \subsection{Simulations analysis}
% %%%%%%%%%%%%%%%%%%%%%%%%%%%%%%%%%%%%

% In this section we will analyze separately the error
% induced by, first the regularization via the kernel $K^\varepsilon$
%   combined with the  particle approximation, secondly by the time
% discretization.

% \subsubsection{Regularization and particle approximation error}

% \label{Simul}
% %For fixed $\varepsilon$ we report on Figure, the MISE~\eqref{eq:MISEApprox} as a function of the number of particles $N$ for different values of the dimension parameter, $d$. 
% %We observe, as expected that the error decreases at a rate of order $N^{-1/2}$ independently of $d$. 

In this subsection, we fix the dimension to $d = 5$. 
We have run a discretized version of the interacting particle system with Euler scheme mesh  $kT/n$ with $n=10$. Notice that this discretization error is neglected in the present analysis. 

Our simulations show that the approximation error presents two types of behavior depending on the number $N$ of particles  
with respect  to the regularization parameter $\varepsilon$.
\begin{enumerate}
\item For large values of $N$, we visualize  a \textit{chaos propagation behavior} for which the  error estimates 
are similar to the ones
provided by  the  density estimation theory~\cite{silverman1986}
%\cite{SilvBook}
  corresponding to the classical framework of independent samples.
\item For small values of $N$ appears  a \textit{transient behavior} for which the bias and variance errors
cannot be  easily described.
\end{enumerate} 

Observe that the  Mean Integrated Squared Error ${\rm MISE}_t(\varepsilon,N):=\E[ \Vert u_t^{\varepsilon,N}-v_t\Vert_2^2 ]$ can be decomposed as the sum of the variance $ V_t(\varepsilon,N)$ and squared bias $B_t^2(\varepsilon,N)$ as follows:
\begin{eqnarray}
\label{eq:VarBias}
{\rm MISE}_t(\varepsilon,N)&=&V_t(\varepsilon,N)+B_t^2(\varepsilon,N) \nonumber \\
%\E \left[ \Vert u_t^{\varepsilon,N}-v_t\Vert_2^2 \right] 
&=& \E\left[ \Vert u_t^{\varepsilon,N}-\E [u_t^{\varepsilon,N}]\Vert_2^2 \right ]
+  \Vert\E[ u_t^{\varepsilon,N}]-v_t\Vert_2^2\ .
\end{eqnarray}
For $N$ large enough, according to Corollary \ref{CCP}, one expects that the propagation of chaos holds.
 Then the particle system $(\xi^{i,N})_{i=1,\cdots, N}$ 
(solution of \eqref{eq:XIi}) is close to an i.i.d. system with common law $m^0$.
% $v^\varepsilon$,the solution of the perturbed PDE~\eqref{eq:pdevEpsilon}.
%the common law of processes $(Y^i)_{i=1,\cdots,N}$ defined as in Theorem \ref{TPC} and satisfying \eqref{eq:xiYuFinalb2}.
 We observe that, in the specific case where the weighting function $\Lambda$ does not depend on the density $u$, 
 for $t \in [0,T]$, we have
 \begin{eqnarray}
\label{eq:BiasApprox}
\E[u_t^{\varepsilon , N}] & = & \frac{1}{N} \E \left[ \sum_{j=1}^N K^{\varepsilon}
( \cdot-\xi^{j,N}_{t})\exp\big \{\int_0^t \Lambda(r(s), \xi^{j,N}_{r(s)})\,ds\big \} \right] \ ,  \nonumber\\
%& \approx & \frac{1}{N} \E \left[ \sum_{j=1}^N  K^{\varepsilon}(y-\xi^{j,N}_t) V_t\big (\xi^{j,N},u^{S^N(\mathbf{\xi})}(\xi^{j,N})\big ) \right] \\
%& \approx & \frac{1}{N} \E \left[ \sum_{j=1}^N  K^{\varepsilon}(y- {Y}^{j}_t) V_t\big ({Y}^{j},u^{S^N(\mathbf{{Y}})}({Y}^{j})\big ) \right]  \\
& = & \E \left[ K^{\varepsilon}( \cdot-Y^{1}_t) V_t\big (Y^{1}\big ) \right] \nonumber \\
& = & u_t^{\varepsilon} \ . %=K^{\varepsilon}\ast v_t^\varepsilon \ .
\end{eqnarray}
Therefore, under the  chaos propagation behavior,
 the approximations below  hold for
 the variance and the squared bias: 
\begin{equation}
\label{eq:aproxVarBias}
V_t(\varepsilon,N)
%:=\E\left[ \Vert u_t^{\varepsilon,N}-\E [u_t^{\varepsilon,N}]\Vert_2^2 \right ]
\approx 
\E\left[ \Vert u_t^{\varepsilon,N}-u_t^{\varepsilon}\Vert_2^2 \right ]
\quad\textrm{and}\quad 
B_t^2(\varepsilon,N) %:=\E\left [ \Vert\E[ u_t^{\varepsilon,N}]-v_t\Vert_2^2 \right ]
\approx  \Vert u_t^{\varepsilon }-v_t\Vert_2^2 \ .
\end{equation}
 We recall that the relation $u^{\varepsilon} = K^{\varepsilon} \ast v^{\varepsilon}$ comes from Theorem 6.1 of \cite{LOR1}, where $v^{\varepsilon}$ is solution of \eqref{epide} with $K = K^{\varepsilon}$.

On Figure~\ref{fig:Variance5d}, we have reported the estimated variance error $V_t(\varepsilon,N)$  as a function of 
the particle number $N$, (on the left graph) and as a function of the regularization parameter $\varepsilon$, (on the right graph),
 for $t = T=1$ and $d=5$.
We have used for this a similar Monte Carlo approximation as \eqref{eq:MISEApprox}.
\\
That figure shows that, when the number of particles is \textit{large enough}, the variance error behaves precisely as in
 the classical case of density estimation encountered in~\cite{silverman1986}, i.e., vanishing at a rate $\frac{1}{N\varepsilon ^d}$, see relation (4.10), Chapter 4., Section 4.3.1.
 This is in particular illustrated by the log-log graphs, showing almost linear curve, when $N$ is sufficiently large. In particular
we observe the following. 
\begin{itemize}
\item On the left graph,  $\log(V_t(\varepsilon,N)) \approx a-\alpha\log N$ with slope $\alpha=1$;
\item On the right graph, $\log V_t(\varepsilon,N) \approx b-\beta \log \varepsilon $  with slope $\beta =5=d$. 
\end{itemize}
It seems that the threshold $N$ after which appears the linear behavior (compatible with the  propagation of chaos situation
corresponding to  asymptotic-i.i.d. particles) decreases
%above which such linear behavior (corresponding to the propagation of chaos situation (asymptotic-i.i.d. particles)) is observed decreases
 when $\varepsilon$ grows. In other words, when $\varepsilon$ is large, less particles $N$ 
are needed to give evidence to the chaotic behavior.

  This phenomenon can be probably explained by analyzing the particle system dynamics. 
%We recall that the relation $u^{\varepsilon} = K^{\varepsilon} \ast v^{\varepsilon}$ comes from Theorem 6.1 of \cite{LOR1}, where $v^{\epsilon}$ is solution of \eqref{epide} with $K = K^{\varepsilon}$.
  Indeed, at each time step,
  the interaction between the particles is due to  the empirical estimation of
 $u^\varepsilon = K^\varepsilon \ast v^\varepsilon $ based on the particle system.
 % at each point of the particles system. 
 Intuitively, the more accurate the approximation $u^{\varepsilon,N}$ of
$u^{\varepsilon}$
is, the less strong the interaction between particles will be.
In the limiting case when  $u^{\varepsilon,N}=u^{\varepsilon}$, the interaction 
disappears.

 Now observe that at time step $0$, the particle system $({\xi}^{i,N}_0)$ is 
i.i.d.  according to $v_0(\cdot)$, so that the estimation of 
$(K^\varepsilon \ast v^\varepsilon)(0,\cdot)$ provided by \eqref{eq:tildeYu} reduces to the classical 
density estimation approach, see \cite{silverman1986}
as mentioned above. 
 %The interaction between particles is then propagated by the succesive estimation of $K^\varepsilon \ast v^\varepsilon$.
  In that classical framework, we emphasize that, for larger values of
 $\varepsilon$, the number of particles, needed to achieve a given density
 estimation accuracy, is smaller. Hence, one can imagine that for larger  
values of $\varepsilon$ less particles will be needed to obtain a quasi-i.i.d 
particle system at time step $1$, $({\xi}^{i,N}_{t_1})$.
 We can then reasonably presume that this initial error propagates 
along the time steps. 

%FIN NEW
%Consequently, the estimation of $K^{\varepsilon}\ast v^\varepsilon$, with a given precision, requires less particles  for large
On Figure~\ref{fig:Bias5d}, we have reported the estimated squared bias error, $B^2_t(\varepsilon,N)$,  as a function of the  regularization parameter, $\varepsilon$, for different values of the particle number $N$, for $t = T=1$ and $d=5$.\\
One can observe that, similarly to the classical i.i.d. case,  (see relation (4.9) in Chapter 4., Section 4.3.1 in \cite{silverman1986}),
 for $N$ large enough, the bias error  does not depend on $N$ and 
 can be  approximated by $a \varepsilon^4$, for some constant $a>0$.
 This is in fact coherent with the bias approximation~\eqref{eq:aproxVarBias}, developed in the specific case where
 the weighting function $\Lambda$ does not depend on the density. 
Assuming the validity of approximation~\eqref{eq:aproxVarBias} and  of the previous empirical observation implies that one can bound the error between the solution, $v^\varepsilon$,
 of the regularized PDE of the form \eqref{epide} (with $K = K^{\varepsilon}$) associated to \eqref{eq:pdev}, and the solution, $v$, of the limit
(non regularized) PDE~\eqref{eq:pdev} as follows 
% \begin{eqnarray}
% \label{eq:vepsilonv}
% \E \Big[ \Vert v_t^\varepsilon-v_t\Vert_2^2 \Big]& \leq & 2 \E \Big[ \Vert v_t^\varepsilon-u_t^{\varepsilon}\Vert_2^2 \Big] + 2 \E \Big[ \Vert u_t^\varepsilon-v_t\Vert_2^2 \Big] \nonumber \\
%   &\leq & 2 \E \Big[ \Vert v_t^\varepsilon-K^\varepsilon\ast v_t^\varepsilon\Vert_2^2 \Big] + 2 \E \Big[ \Vert u_t^\varepsilon-v_t\Vert_2^2 \Big] \nonumber \\
% &\leq & 2(a'+a)\varepsilon^4.
% \end{eqnarray}
\begin{eqnarray}
\label{eq:vepsilonv}
 \Vert v_t^\varepsilon-v_t\Vert_2^2 & \leq & 2  \Vert v_t^\varepsilon-u_t^{\varepsilon}\Vert_2^2  + 2  \Vert u_t^\varepsilon-v_t\Vert_2^2  \nonumber \\
  &\leq & 2  \Vert v_t^\varepsilon-K^\varepsilon\ast v_t^\varepsilon\Vert_2^2 
 + 2  \Vert u_t^\varepsilon-v_t\Vert_2^2  \nonumber \\
&\leq & 2(a'+a)\varepsilon^4.
\end{eqnarray}
Indeed, at least, the first  term in  the second line can be easily bounded, supposing that  $v_t^\varepsilon$ has
a bounded second derivative.
This constitutes an empirical proof of the fact  that $v^\varepsilon$ converges to $v$.\\
As observed in the variance error graphs, the threshold $N$, above which the propagation of chaos behavior is observed decreases with $\varepsilon$. 
Indeed, for $\varepsilon>0.6$ we observe a chaotic behavior of the bias error, starting from $N\geq 500$, whereas for
 $\varepsilon\in [0.4,0.6]$, this chaotic behavior appears only for $N\geq 5000$. \\
For small values of $\varepsilon\leq 0.6$, the bias highly depends on $N$ for any $N\leq 10^4$; moreover that dependence
becomes less relevant when  $N$ increases.  This is probably due to the combination of two effects: 
the lack of chaos propagation phenomenon and the fact that the coefficient $\Lambda$ depends on $u$, so that
\eqref{eq:BiasApprox} does not hold in that context.
 
%Indeed one can notice that for a given value of $N$, the quadratic fit $a\varepsilon^4+b$ becomes to be valid for large enough $\varepsilon$. 

%Besides, for small values of $N$, one can observe that the bias error 

Taking into account both  the bias and the variance error in the MISE \eqref{eq:VarBias}, the choice of $\varepsilon$ has to
 be carefully optimized w.r.t. the number of particles:  $\varepsilon$ going to zero together with $N$ going to infinity at a 
judicious relative rate seem to ensure the convergence of the estimated MISE to zero. 
This kind of tradeoff is standard in density estimation theory and was already investigated theoretically in the context of
 forward interacting particle systems related to conservative regularized nonlinear PDE in
 \cite{JourMeleard}. Extending this type of theoretical analysis to our non conservative framework is beyond the scope
 of the present paper.
% and should be the subject of future works. 

\begin{figure}[!h]
\begin{center}
\subfigure[Variance as a function of $N$]
{\includegraphics[width=0.49\linewidth]{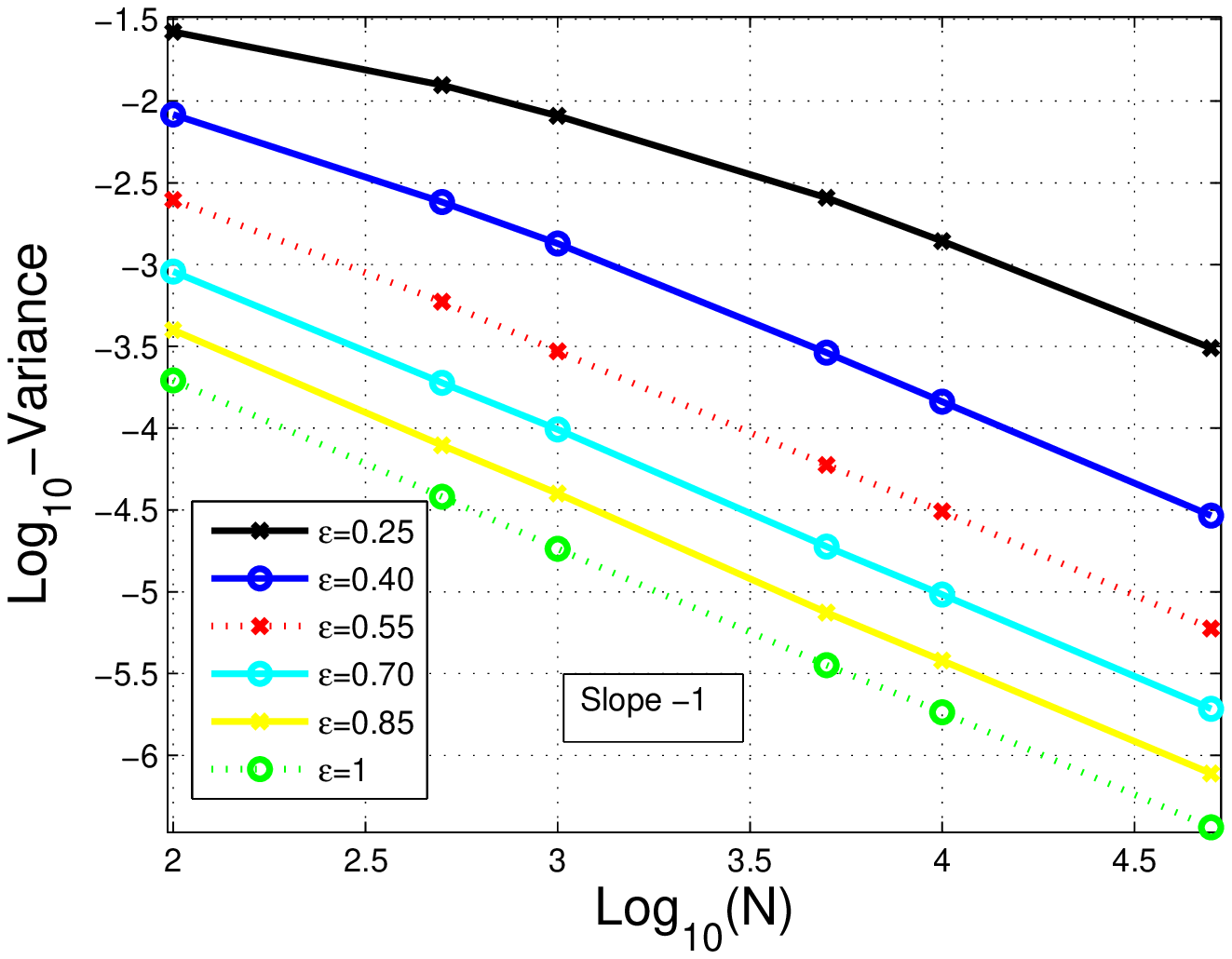}}
\subfigure[Variance as a function of $\varepsilon$]{\includegraphics[width=0.49\linewidth]{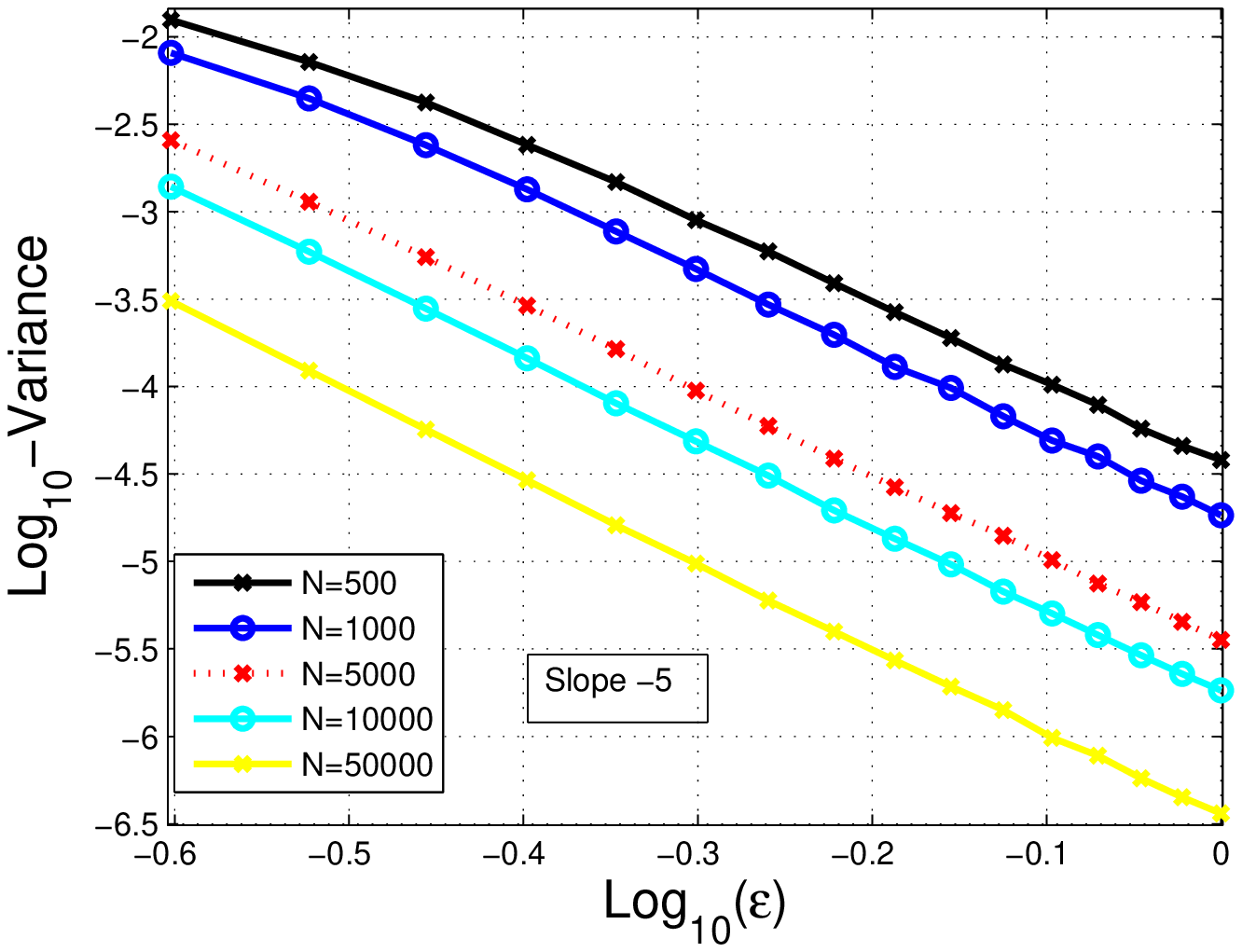}}
\end{center}
\caption{{\small  Variance error as a function of the number of particles, $N$, and the mollifier window width, $\varepsilon$, for dimension $d=5$ at the final time step $T=1$. } }
\label{fig:Variance5d}
\end{figure}

\begin{figure}[!h]
\begin{center}
{\includegraphics[width=0.8\linewidth,height=7cm]{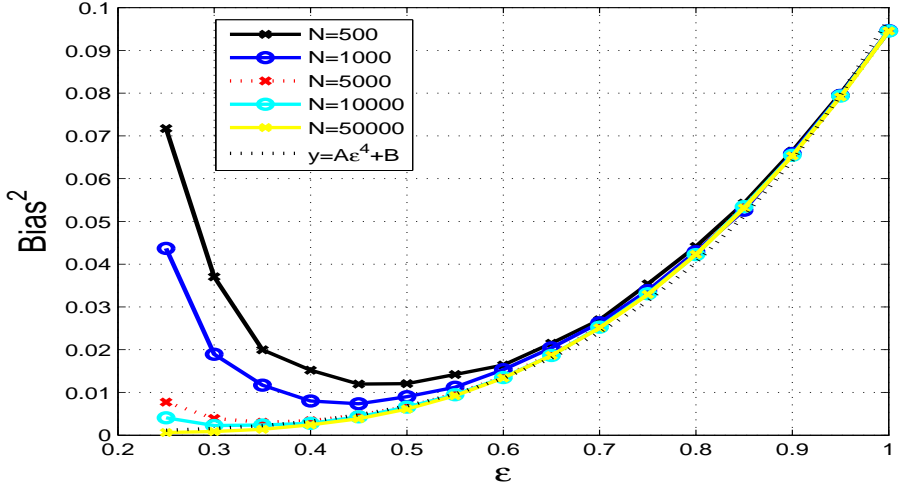}}
\end{center}
\caption{{\small  Bias error as a function of the mollifier window width, $\varepsilon$, for dimension $d=5$ at the final time step $T=1$. } }
\label{fig:Bias5d}
\end{figure}

\subsubsection{Time discretization error}
%%%%%%%%%%%%%%%%%%%%%%%%%%%%%
%DEBUT MODIFS NADIA 10/06/2016
%%%%%%%%%%%%%%%%%%%%%%%%%%%%%

In this subsection, we are interested in analysing via numerical simulations the time discretization error w.r.t. to $\delta t ={T}/{n}$. As announced in Remark~\ref{eq:R68}, we suspect that the rate in \eqref{eq:CS8} is not optimal and that the MISE error induced by the time discretization is of order  $1/{n}^2$ instead of $1/{n}$. 

Let $\tilde u^{\varepsilon, N, n}_T$ denote the particle approximation
% $\tilde u_T$,
obtained by scheme~\eqref{eq:tildeYu} with a number of particles, $N$, a regularization parameter, $\varepsilon$, and a number of time steps, $n$. 
In order to focus on the time discretization error apart from the particle approximation  and the regularization error 
(related to $N$ and   $\varepsilon$), we have considered  errors of the type 
$\E [\Vert \tilde u^{\varepsilon, N, n}_T-\tilde u_T^{\varepsilon, N, n_0}\Vert^2_2]$ for different numbers of time steps $n<n_0$  where $n_0$ is supposed to be a large number of time steps. 
More precisely, we have decomposed this error into a variance and a squared bias term as 
%VERIFIER AVEC NADIA
\begin{eqnarray} \label{EBiaisVar}
\E [\Vert \tilde u^{\varepsilon, N, n}_T-\tilde u_T^{\varepsilon, N, n_0}\Vert^2_2] &=&
\underbrace{\E \big [\Vert \tilde u^{\varepsilon, N, n}_T-\E[\tilde u_T^{\varepsilon, N, n}]\Vert^2_2\big ] 
+ \E \big [\Vert \tilde u^{\varepsilon, N, n_0}_T-\E [\tilde u_T^{\varepsilon, N, n_0}]\Vert^2_2\big ]}_{\textrm{Variance}} \nonumber \\  && \\
 &+& \underbrace{\Vert \E[\tilde u^{\varepsilon, N, n}_T]-\E [\tilde u_T^{\varepsilon, N, n_0}]\Vert^2_2}_{\textrm{Bias}^2}, \nonumber
\end{eqnarray}
 if 
$u^{\varepsilon, N, n}_T$ and $u^{\varepsilon, N, n_0}_T $ are independent.

On Figure~\ref{fig:MISEn}, we have reported the Monte Carlo estimation (according to~\eqref{eq:MISEApprox}, with $Q=1000$ runs) of the above variance and squared bias terms in a log-log scale in order to diagnose the expected rate of convergence $1/{n}^2$ via a straight line with slope $-2$. All the parameters are similar to the simulations performed in  previous subsection excepted for the dimension $d=1, N = 5000$ and $n_0$ is set to $1000$ time steps. One can observe that the variance term (in dashed lines) seems not to depend on the number of time steps $n$ whereas the squared bias term decreases as expected at a rate close to $1/n^2$. 
\begin{figure}[!h]
\begin{center}
{\includegraphics[width=0.8\linewidth,height=7cm]{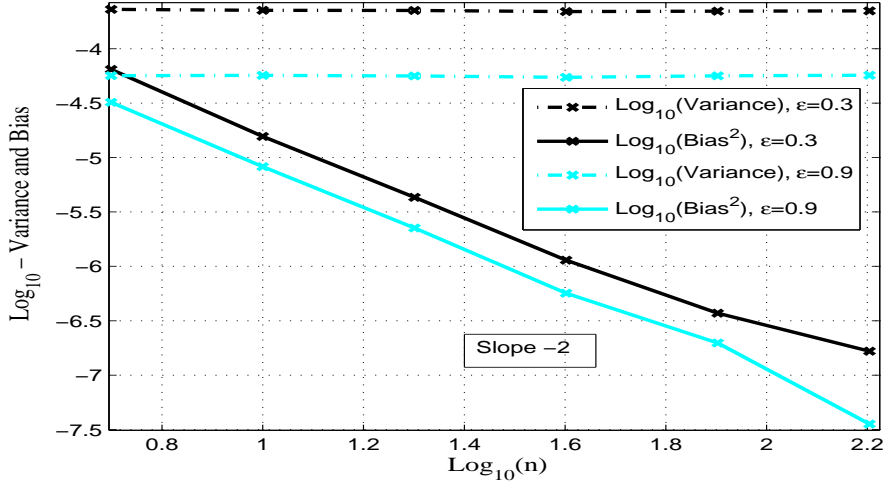}}
\end{center}
\caption{{\small  Variance and squared bias error \eqref{EBiaisVar} as a function of the number of time steps, $n=5,10,20,40,80,160,320$, at the final time step $T=1$ (with dimension $d=1$, $N=5000$ particles and $\varepsilon=0.3$ or $0.9$). } }
\label{fig:MISEn}
\end{figure}

%%%%%%%%%%%%%%%%%%%%%%%%%%%%%
%FIN MODIFS NADIA 10/06/2016
%%%%%%%%%%%%%%%%%%%%%%%%%%%%%

\section{Appendix}

\setcounter{equation}{0}
\label{SAppendix}

In this appendix, we present the proof of Lemma \ref{lem:Discrete}.
%We suppose here the vaAssumption 2. 
We first proceed with the proof of some intermediary inequalities. 
%before to prove Lemma \ref{lem:Discrete}.
\begin{lem} We suppose Assumption \ref{ass:main}.
\label{lem:ViVi'}
Let $N \in \N^{\star}$. Let $(\xi^{i,N})_{i=1,\cdots,N}$ be (a solution of) the interacting particle system \eqref{eq:XIi};
 let $(\tilde{\xi}^{i,N})_{i=1,\cdots,N}$ and $\tilde{u}$  as defined as in  the discretized interacting particle system \eqref{eq:tildeYu}. \\
%Under the same assumptions as in Proposition \ref{prop:DiscretTime}, 
The random variables $V_t^i := e^{\int_0^t\Lambda\big (s,\tilde \xi^{i,N}_{s},u^{S^N(\tilde{\xi})}_{s}(\tilde \xi^{i,N}_{s})\big )ds}$ and $ \tilde V_t^i := e^{\int_0^t\Lambda\big (r(s),\tilde \xi^{i,N}_{r(s)},\tilde{u}_{r(s)}(\tilde \xi^{i,N}_{r(s)})\big )ds}$,
 for all  $t \in [0,T]$, $i \in \{ 1,\cdots,N\}$ fulfill the following.
\begin{enumerate}
\item For all $t \in [0,T]$, $i \in \{1,\cdots,N\}$
\begin{eqnarray}
\label{E941}
 \E[\vert \tilde V_t^i-V_t^i \vert^2]
\leq C\delta t + C\E \left[ \int_0^t \vert \tilde{\xi}^{i,N}_{r(s)} - \tilde{\xi}^{i,N}_{s}  \vert^2   \ ds \right] + C \E \left[ \int_0^t \vert \tilde{u}_{r(s)}(\tilde \xi^{i,N}_{r(s)})-u^{S^N(\tilde \xi)}_s(\tilde \xi^{i,N}_s)\vert^2 ds \right],
\end{eqnarray}
where $C$ is a real positive constant depending only on $M_{\Lambda}$, $L_{\Lambda}$ and $T$.
\item For all $(t,y) \in [0,T] \times \R^d$, $i \in \{1,\cdots,N\}$
\begin{eqnarray}
\label{E942}
\vert \tilde{u}_{t}(y)-u_t^{S^N(\tilde \xi)}(y)\vert^2 \leq  \frac{M_K}{N}\sum_{i=1}^N K(y-\tilde \xi^{i,N}_t) \;  \vert \tilde V_t^i-V_t^i \vert^2 \ .
\end{eqnarray}
\end{enumerate}
\end{lem} 
\begin{proof}[Proof of Lemma \ref{lem:ViVi'}]
Let us fix $t \in [0,T]$, $i \in \{1,\cdots,N\}$. To prove \eqref{E941}, it is enough to recall that $\Lambda$ being Lipschitz w.r.t. the space variables and $\frac{1}{2}$-Holder continuous w.r.t. the time variable, the inequality~\eqref{EMajor1} yields
\begin{equation}
\label{eq:VtildeV2}
 \vert \tilde V_t^i
-
V_t^i \vert^2
\leq 
3e^{2tM_{\Lambda}}L^2_{\Lambda} \int_0^t \left [\vert r(s)-s\vert +\vert \tilde \xi^{i,N}_{r(s)}-\tilde \xi^{i,N}_s\vert^2+\vert \tilde{u}_{r(s)}(\tilde \xi^{i,N}_{r(s)})-u_s^{S^N(\tilde \xi)}(\tilde \xi^{i,N}_s)\vert^2 \right ]\,ds\ ,
\end{equation}
and taking the expectation in both sides of \eqref{eq:VtildeV2} implies \eqref{E941} with $C := 3e^{2TM_{\Lambda}}L^2_{\Lambda}$. \\
% which concludes the proof of \eqref{E941}. \\
Let us fix $y \in \R^d$. Concerning \eqref{E942}, by recalling the third line equation of \eqref{eq:tildeYu} and the linking equation \eqref{NSDE3} (with $m = S^N(\tilde{\xi})$), we have 
\begin{eqnarray}
\vert \tilde{u}_{t}(y)-u_t^{S^N(\tilde \xi)}(y)\vert^2 & = & \left \vert \frac{1}{N} \sum_{i=1}^N K(y- \tilde{\xi}_t^{i,N}) \tilde{V}^i_t - \frac{1}{N} \sum_{i=1}^N K(y- \tilde{\xi}_t^{i,N}) V^i_t   \right \vert^2 \nonumber \\
& = & \left \vert \frac{1}{N} \sum_{i=1}^N K(y- \tilde{\xi}_t^{i,N}) \left( \tilde{V}^i_t - V^i_t \right)  \right \vert^2 \nonumber \\
&\leq &
\label{E944}
\frac{1}{N}\sum_{i=1}^N K^2(y-\tilde \xi^{i,N}_t)  \vert \tilde V_t^i
-
V_t^i \vert^2 \nonumber\\
& \leq & \frac{M_K}{N}\sum_{i=1}^N  K(y-\tilde \xi^{i,N}_t) \;  \vert \tilde V_t^i-V_t^i \vert^2 \ ,
\end{eqnarray}
which concludes the proof of \eqref{E942} and therefore of  Lemma \ref{lem:ViVi'}.
\end{proof}
%From now on, we prove Lemma \ref{lem:Discrete}.
\begin{proof}[\bf Proof of Lemma \ref{lem:Discrete}.]
All along this proof, $C$ will denote a positive constant that only  depends \\
 $T,M_K,m_{\Phi},m_g,L_K,L_\Phi,L_g$ and $M_{\Lambda},L_{\Lambda}$ and that can change from line to line. \\
%ANTHONY. ENCORE LES AUTRES CONSTANTES?\\
 Let us fix $t \in [0,T]$.

\begin{itemize}
\item Inequality \eqref{eq:vtilde} of Lemma~\ref{lem:Discrete} is simply a consequence of the following computation:
%the fact that the coefficients $\Phi$ and $g$ are uniformly bounded. Indeed,
\begin{eqnarray*}
\E[\vert \tilde \xi ^{i,N}_{r(t)}-\tilde \xi^{i,N}_{t}\vert^2]&=&\E\left [\Big \vert \int_{r(t)}^{t}\Phi(r(s),\tilde{\xi}^{i,N}_{r(s)},\tilde{u}_{r(s)}(\tilde \xi ^{i,N}_{r(s)}))\,dW_s+\int_{r(t)}^{t}g(r(s),\tilde{\xi}^{i,N}_{r(s)},\tilde{u}_{r(s)}(\tilde \xi ^{i,N}_{r(s)}))\,ds\Big \vert^2 \right ]\\
&\leq &
4\E\left [ \int_{r(t)}^{t}\vert \Phi(r(s),\tilde{\xi}^{i,N}_{r(s)},\tilde{u}_{r(s)}(\tilde \xi ^{i,N}_{r(s)})) - \Phi(r(s),0,0) \vert ^2\,ds\right ] + 4\E\left [ \int_{r(t)}^{t} \vert \Phi(r(s),0,0) \vert ^2\,ds\right ] \\
&& \; + 4(t-r(t))\E \left [\int_{r(t)}^{t}\vert g(r(s),\tilde{\xi}^{i,N}_{r(s)},\tilde{u}_{r(s)}(\tilde \xi ^{i,N}_{r(s)})) - g(r(s),0,0)\vert ^2\,ds \right ] \\
&& \; + 4(t-r(t))\E\left [ \int_{r(t)}^{t} \vert g(r(s),0,0) \vert ^2\,ds\right ] \\
& \leq & 8(L_{\Phi}^2+(t-r(t))L_{g}^2)\int_{r(t)}^t \E \Big[\vert \tilde{\xi}^{i,N}_{r(s)} \vert^2 \Big] + \E \Big[\vert \tilde{u}_{r(s)}(\tilde \xi ^{i,N}_{r(s)}) \vert^2 \Big] ds \\
&& + \; 4(t-r(t)) \Big( \sup_{s \in [0,T]} \vert \Phi(s,0,0) \vert^2 + (t-r(t)) \sup_{s \in [0,T]} \vert g(s,0,0) \vert^2 \Big)\\
&\leq & C\delta t \ ,\quad \textrm{as soon as }\quad \delta t \in\, ]0, 1[ \ ,
\end{eqnarray*}
where we have used the fact, under items 1. and 6. of Assumption \ref{ass:main2}, that the second order moment of $\tilde{\xi}^{i,N}_s$ is uniformly bounded. $\Lambda$ being uniformly bounded (item 3. of Assumption \ref{ass:main2}), the function $\tilde{u}$ as well. We have finally invoked item 6. of Assumption \ref{ass:main2}.
\item Now, let us focus on the second inequality \eqref{E82} of Lemma~\ref{lem:Discrete}. Note that for any $y\in \R^d$, the following inequality holds:
\begin{eqnarray}
\label{eq:r(t)t}
\vert \tilde{u}_{r(t)}(y)-\tilde{u}_t(y)\vert & \leq & \frac{1}{N} \sum_{i=1}^N \left[ \left \vert K(y-\tilde \xi^{i,N}_{r(t)}) -K(y-\tilde \xi^{i,N}_t)\right \vert \, e^{\int_0^{r(t)}\Lambda\big (r(s),\tilde \xi^{i,N}_{r(s)},\tilde{u}_{r(s)}(\tilde \xi^{i,N}_{r(s)})\big )ds} \right . \nonumber \\
& &  +
%\frac{1}{N}\sum_{i=1}^N
 \left .  K(y-\tilde \xi^{i,N}_t) \,\left \vert e^{\int_0^{r(t)}\Lambda\big (r(s),\tilde \xi^{i,N}_{r(s)},\tilde{u}_{r(s)}(\tilde \xi^{i,N}_{r(s)})\big )ds}- e^{\int_0^t\Lambda\big (r(s),\tilde \xi^{i,N}_{r(s)},\tilde{u}_{r(s)}(\tilde \xi^{i,N}_{r(s)})\big )ds} \right \vert \right ] \ . \nonumber \\ 
\end{eqnarray}
Using the fact that  $K$ and $\Lambda$ are bounded, 
one can apply~\eqref{EMajor1} to bound the second term of the sum on the r.h.s. of the above inequality as follows:
\begin{eqnarray}
\label{eq:K_exp}
K(y-\tilde \xi^{i,N}_t) \,\left \vert e^{\int_0^{r(t)}\Lambda\big (r(s),\tilde \xi^{i,N}_{r(s)},\tilde{u}_{r(s)}(\tilde \xi^{i,N}_{r(s)})\big )ds}- e^{\int_0^t\Lambda\big (r(s),\tilde \xi^{i,N}_{r(s)},\tilde{u}_{r(s)}(\tilde \xi^{i,N}_{r(s)})\big )ds}\right \vert 
& \leq & M_Ke^{tM_{\Lambda}}(t-r(t))M_{\Lambda} \nonumber \\
& \leq & C\delta t \ .
\end{eqnarray}
The first term of the sum on the r.h.s. of~\eqref{eq:r(t)t} is bounded using the Lipschitz property of $K$ and the fact that $\Lambda$ is bounded. 
\begin{eqnarray}
\label{eq:KKtilde}
\left \vert K(y-\tilde \xi^{i,N}_{r(t)}) -K(y-\tilde \xi^{i,N}_t)\right \vert\, e^{\int_0^{r(t)}\Lambda\big (r(s),\tilde \xi^{i,N}_{r(s)},\tilde{u}_{r(s)}(\tilde \xi^{i,N}_{r(s)})\big )ds} \leq  
L_Ke^{tM_{\Lambda}} \vert \tilde \xi^{i,N}_{r(t)}-\tilde \xi^{i,N}_t\vert \ .
\end{eqnarray}
Injecting \eqref{eq:K_exp}  and \eqref{eq:KKtilde} in \eqref{eq:r(t)t},  for all $y \in \R^d$, we obtain
$$
\vert \tilde{u}_{r(t)}(y) - \tilde{u}_{t}(y) \vert \leq C \delta t + \frac{L_Ke^{tM_{\Lambda}}}{N} \sum_{i=1}^N \vert \tilde{\xi}^{i,N}_{r(t)}-\tilde{\xi}^{i,N}_{t} \vert,
$$
which finally implies that 
%for any $y \in \R^d$
$$
\Vert \tilde{u}_{r(t)}-\tilde{u}_t \Vert_{\infty}^2\leq C\delta t ^2+\frac{C}{N}\sum_{i=1}^N \vert \tilde \xi^{i,N}_{r(t)}-\tilde \xi^{i,N}_t\vert^2 \ .
$$
We conclude by using  inequality \eqref{eq:vtilde} of Lemma~\ref{lem:Discrete} after taking the expectation of the r.h.s. of the above inequality. 

\item Finally, we  deal with  inequality \eqref{E83} of Lemma~\ref{lem:Discrete}. Observe that the 
error on the left-hand side can be decomposed as  
\begin{eqnarray}
\label{eq:tildev1}
\E[\Vert \tilde{u}_{r(t)}-u_t^{S^N(\tilde \xi)}\Vert ^2_{\infty}]
&\leq &2\E[\Vert \tilde{u}_{r(t)}-\tilde{u}_t\Vert^2_{\infty}]+2\E[\Vert \tilde{u}_{t}-u_t^{S^N(\tilde \xi)}\Vert^2_{\infty}] \nonumber \\
&\leq & C\delta t +2\E[\Vert \tilde{u}_{t}-u_t^{S^N(\tilde \xi)}\Vert^2_{\infty}]\ ,
\end{eqnarray}
where we have used  inequality \eqref{E82} of Lemma \ref{lem:Discrete}. \\
Let us consider the second term on the r.h.s. of the above inequality. 
To simplify the notations, we introduce the real valued random variables 
\begin{equation}
\label{eq:VVtildeDef2}
V_t^i := e^{\int_0^t\Lambda\big (s,\tilde \xi^{i,N}_{s},u^{S^N(\tilde{\xi})}_{s}(\tilde \xi^{i,N}_{s})\big )ds}\quad \textrm{and}\quad \tilde V_t^i := e^{\int_0^t\Lambda\big (r(s),\tilde \xi^{i,N}_{r(s)},\tilde{u}_{r(s)}(\tilde \xi^{i,N}_{r(s)})\big )ds}\ ,
\end{equation}
defined for any $i=1,\cdots N$ and $t\in [0,T]$. \\ 
Using successively inequalities \eqref{E941} of Lemma \ref{lem:ViVi'}, \eqref{eq:vtilde} of~Lemma~\ref{lem:Discrete} and \eqref{eq:uu'} of Proposition \ref{lem:uu'}, we have for all $i \in \{1,\cdots,N\}$,   
\begin{eqnarray}
\label{eq:Vnew}
 \E[\vert \tilde V_t^i-V_t^i \vert^2]
&\leq & C\delta t + C \E \left[ \int_0^t \vert \tilde{u}_{r(s)}(\tilde \xi^{i,N}_{r(s)})-u^{S^N(\tilde \xi)}_s(\tilde \xi^{i,N}_s)\vert^2 ds \right] \nonumber\\
& \leq & C\delta t + C \E \left[ \int_0^t \vert \tilde{u}_{r(s)}(\tilde \xi^{i,N}_{r(s)})-u^{S^N(\tilde \xi)}_s(\tilde \xi^{i,N}_{r(s)})\vert^2 ds \right] \nonumber \\
&& + \; C \E \left[ \int_0^t \vert u_{s}^{S^N(\tilde{\xi})}(\tilde \xi^{i,N}_{r(s)})-u^{S^N(\tilde \xi)}_s(\tilde \xi^{i,N}_{s})\vert^2 ds \right] \nonumber \\
&\leq & C\delta t +C\int_0^t \left [ \E [\Vert \tilde{u}_{r(s)}-u^{S^N(\tilde \xi)}_s\Vert_\infty^2] + \E[\vert \tilde \xi^{i,N}_{r(s)}-\tilde \xi^{i,N}_s\vert^2] \right ]\, ds \nonumber\\
&\leq &
C\delta t +C\int_0^t \E [\Vert \tilde{u}_{r(s)}-u^{S^N(\tilde \xi)}_s\Vert_\infty^2] \, ds \ .
\end{eqnarray}
On the other hand, inequality \eqref{E942} of Lemma \ref{lem:ViVi'} implies
\begin{eqnarray}
\label{major_E86}
\Vert  \tilde{u}_{t}-u_t^{S^N(\tilde \xi)} \Vert_{\infty}^2 \leq \frac{M_K^2}{N}\sum_{i=1}^N   \vert \tilde V_t^i-V_t^i \vert^2 \ .
\end{eqnarray}
Taking the expectation in both sides of \eqref{major_E86} and using \eqref{eq:Vnew} give
\begin{equation}
\label{eq:tildevxi}
\E[\Vert \tilde{u}_{t}-u_t^{S^N(\tilde \xi)}\Vert^2_{\infty}]
%&\leq & C\delta t +\frac{C}{N}\sum_{i=1}^N \E \int_0^t \vert \tilde{u}_{r(s)}(\tilde \xi^{i,N}_{r(s)})-u^{S^N(\tilde \xi)}_s(\tilde \xi^{i,N}_s)\vert^2 ds \nonumber\\
%&\leq & C\delta t +C\int_0^t \left [ \E [\Vert \tilde{u}_{r(s)}-u^{S^N(\tilde \xi)}_s\Vert_\infty^2] + \sup_{i=1,\cdots N}\E[\vert \tilde \xi^{i,N}_{r(s)}-\tilde \xi^{i,N}_s\vert^2] \right ]\, ds \nonumber\\
\leq \frac{M_K^2}{N}\sum_{i=1}^N   \E \left[ \vert \tilde V_t^i-V_t^i \vert^2 \right] \leq  C\delta t +C\int_0^t \E [\Vert \tilde{u}_{r(s)}-u^{S^N(\tilde \xi)}_s\Vert_\infty^2] \, ds\ .
\end{equation}
We end the proof by injecting this last inequality in~\eqref{eq:tildev1} and by applying Gronwall's lemma.
\end{itemize}
\end{proof}

\noindent\textbf{ACKNOWLEDGEMENTS.}
The authors are very grateful to the anonymous Referee for her / his careful reading of the paper
and the suggestions which have largely contributed to improve the first submitted version. 
The third named   author has benefited partially from the
support of the ``FMJH Program Gaspard Monge in optimization and operation
research'' (Project 2014-1607H).

\newpage
% % %\addcontentsline{toc}{section}{Bibliography}
\bibliographystyle{plain}
\bibliography{NonConservativePDE}

\end{document}